\documentclass[a4,11pt]{article}
\usepackage{amsthm}
\usepackage[colorlinks=true,citecolor=black,linkcolor=black,urlcolor=blue]{hyperref}

\usepackage[a4paper]{geometry}
\makeatletter
\def\gnewcommand{\g@star@or@long\new@command}
\def\grenewcommand{\g@star@or@long\renew@command}
\def\g@star@or@long#1{%
  \@ifstar{\let\l@ngrel@x\global#1}{\def\l@ngrel@x{\long\global}#1}}
\makeatother

\newcommand{\newnumbered}[2]{\newtheorem{#1}[theorem]{#2}}
\newcommand{\newunnumbered}[2]{\newtheorem{#1}[theorem]{#2}}
\newcommand{\Title}[2]{\title{#1}\gnewcommand{\Acknowledgements}{\section*{Acknowledgements} #2}}
\newcommand{\Author}[2][]{\author{#2}}
\newcommand{\Comma}{\and}
\newcommand{\Und}{\and}
\newcommand{\br}{, }
\newcommand{\fs}{. }
\newcommand{\sep}{, }
\newcommand{\thanksone}[3][]{#1\thanks{#3\email{\tt #2}}}
\newcommand{\thankstwo}[4][]{#1\thanks{#3} \thanks{#4\email{\tt #2}}}
\newcommand{\thanksthree}[3][]{#1\thanks{#3\email{\tt #2}}}
\newcommand{\email}[1]{#1}
\newcommand{\classno}[2][2000]{}
\newcommand{\printscl}{}
\newcommand{\mktitle}{\maketitle}
\newcommand{\keywords}[1]{}
\newcommand{\mkabstitle}{}

\newcommand{\bqed}{}
\newcommand{\aem}{\em}

\newenvironment{frmatter}{}{}

\usepackage[utf8x]{inputenc}
\usepackage{amsfonts}
\usepackage{amsmath}
\usepackage{amssymb}
\usepackage{mathtools}
\usepackage[mathscr]{eucal}
\usepackage{color}
\usepackage{array}
\usepackage{afterpage}
\usepackage{listings}
\renewcommand{\arraystretch}{1.15}
\usepackage{comment}

\newcommand{\noop}[1]{}
\newcommand{\R}{\ensuremath{\mathbb{R}}}
\newcommand{\GQ}{\ensuremath{\operatorname{GQ}}}
\newcommand{\mm}{\ensuremath{\!-\!}}
\newcommand{\pp}{\ensuremath{\!+\!}}
\newcommand{\ee}{\ensuremath{\!=\!}}
\newcommand{\pty}[1]{{\em #1}}

\newcommand{\Tr}{\ensuremath{\operatorname{Tr}}}

\newcommand{\diag}{\ensuremath{\operatorname{diag}}}
\newcommand{\sV}[2]{{
	\setlength{\arraycolsep}{2pt}
	\renewcommand{\arraystretch}{0.8}
	\left[\begin{array}{ccc} #1 \\ #2 \end{array}\right]
}}
\newcommand{\sW}[2]{{
	\setlength{\arraycolsep}{2pt}
	\renewcommand{\arraystretch}{0.8}
	\left[\begin{array}{cccc} #1 \\ #2 \end{array}\right]
}}

\newcommand{\nklm}{\ensuremath{(n, k, \lambda, \mu)}}
\newcommand{\willif}[2][]{\ensuremath{\langle #2 \rangle}#1}

\DeclareFixedFont{\ttb}{T1}{txtt}{bx}{n}{10} 
\DeclareFixedFont{\ttm}{T1}{txtt}{m}{n}{10}  

\newcommand{\pythonstyle}{\lstset{
language=Python,
basicstyle=\ttm,
otherkeywords={self},             
keywordstyle=\ttb\color{deepblue},
emph={sage,__init__,True,False,None},  
emphstyle=\ttb\color{deepred},    
stringstyle=\color{deepgreen},
frame=tb,                         
showstringspaces=false
}}

\lstnewenvironment{python}[1][]
{
\pythonstyle
\lstset{#1}
}
{}

\newcommand\pythoninline[1]{{\pythonstyle\lstinline!#1!}}

\newtheorem{theorem}{Theorem}[section] 
\newtheorem{corollary}[theorem]{Corollary}
\newtheorem{proposition}[theorem]{Proposition}

\newnumbered{assertion}{Assertion}    
\newnumbered{conjecture}{Conjecture}  
\newnumbered{definition}{Definition}
\newnumbered{hypothesis}{Hypothesis}
\newnumbered{remark}{Remark}
\newnumbered{note}{Note}
\newnumbered{observation}{Observation}
\newnumbered{problem}{Problem}
\newnumbered{question}{Question}
\newnumbered{algorithm}{Algorithm}
\newnumbered{example}{Example}
\newunnumbered{notation}{Notation} 
\numberwithin{equation}{section}

\begin{document}
\begin{frmatter}

\Title{On few-class $Q$-polynomial association schemes: feasible parameters and nonexistence results}{%
Alexander Gavrilyuk is supported by
BK21plus Center for Math Research and Education at Pusan National University,
and by Basic Science Research Program through the National Research Foundation of Korea (NRF) funded
by the Ministry of Education (grant number NRF-2018R1D1A1B07047427).
Jano\v{s} Vidali is supported by the Slovenian Research Agency
(research program P1-0285 and project J1-8130).
Alexander Gavrilyuk and Jano\v{s} Vidali
are also jointly supported by the Slovenian Research Agency
(Slovenia-Russia bilateral grant number BI-RU/19-20-007).
Jason Williford was supported by
National Science Foundation (NSF) grant DMS-1400281.
}

\Author[Alexander L.~Gavrilyuk, Jano\v{s} Vidali and Jason S.~Williford]{%
\thanksone[Alexander L.~Gavrilyuk]{alexander.gavriliouk@gmail.com}{%
Center for Math Research and Education\br
Pusan National University\br
2, Busandaehak-ro 63beon-gil\br
Geumjeong-gu, Busan, 46241\br
Republic of Korea\fs
}%
\Comma
\thankstwo[Jano\v{s} Vidali]{janos.vidali@fmf.uni-lj.si}{%
Faculty of Mathematics and Physics\br
University of Ljubljana\br
Jadranska ulica 21\br
1000 Ljubljana\br
Slovenia\fs
}{%
Institute of Mathematics, Physics and Mechanics\br
Jadranska ulica 19\br
1000 Ljubljana\br
Slovenia\fs
}
\Und
\thanksthree[Jason S.~Williford]{jwillif1@uwyo.edu}{%
Department of Mathematics and Statistics\br
University of Wyoming\br
1000 E.~University Ave.\br
Laramie, WY 82071\br
United States of America\fs
}
}

\classno{05E30 (primary), 05B15 (secondary)}

\date{\today}

\mktitle

\keywords{%
association scheme\sep
$Q$-polynomial\sep
feasible parameters\sep
distance-regular graph}

\begin{abstract}
We present the tables of feasible parameters of primitive
$3$-class $Q$-polynomial association schemes
and $4$- and $5$-class $Q$-bipartite association schemes
(on up to $2800$, $10000$, and $50000$ vertices, respectively),
accompanied by a number of nonexistence results for such schemes
obtained by analysing triple intersection numbers of putative open cases.
\printscl
\end{abstract}

\end{frmatter}

\mkabstitle

\section{Introduction}

Much attention in literature on association schemes
has been paid to distance-regular graphs,
in particular to those of diameter $2$,
also known as strongly regular graphs
-- however, their complete classification is still a widely open problem.
The tables of their feasible parameters,
maintained by A.~E.~Brouwer~\cite{Bsrg,BCN},
are very helpful for the algebraic combinatorics community,
in particular when one wants to check
whether a certain example has already been proven
(not) to exist, to be unique, etc.
Compiling such a table can be a challenging problem, as, for example,
some feasibility conditions require calculating
roots of high degree polynomials.

The goal of this work is to present the tables of feasible parameters
of $Q$-polynomial association schemes, compiled by the third author,
and accompanied by a number of nonexistence results obtained
by the first two authors.

Recall that $Q$-polynomial association schemes can be seen
as a counterpart of distance-regular graphs,
which, however, remains much less explored,
although they have received considerable attention
in the last few years~\cite{DMM,K,MMW,MT}
due to their connection with some objects in quantum information theory
such as equiangular lines and real mutually unbiased bases~\cite{KPhD}.

More precisely, let $A_0,\ldots,A_D$ and $E_0,\ldots,E_D$
denote the adjacency matrices and the primitive idempotents
of an association scheme, respectively.
An association scheme is {\em $P$-polynomial} (or {\em metric}) if,
after suitably reordering the relations,
there exist polynomials $v_i$ of degree $i$
such that $A_i = v_i(A_1)$ ($0 \le i \le D$).
If this is the case, the matrix $A_i$ can be seen
as the distance-$i$ adjacency matrix of a distance-regular graph
and vice-versa.
Similarly, an association scheme is $Q$-polynomial (or {\em cometric}) if,
after suitably reordering the eigenspaces,
there exist polynomials $v^*_j$ of degree $j$
such that $E_j = v^*_j(E_1)$ ($0 \le j \le D$),
where the matrix multiplication is entrywise.
These notions are due to Delsarte~\cite{D},
who introduced the $P$-polynomial property
as an algebraic definition of association schemes
generated by distance-regular graphs,
and then defined $Q$-polynomial association schemes
as the dual concept to $P$-polynomial association schemes.

Many important examples of $P$-polynomial association schemes,
which arise from classical algebraic objects
such as dual polar spaces and forms over finite fields,
also possess the $Q$-polynomial property.
Bannai and Ito~\cite{BI} posed the following conjecture:

\begin{conjecture}\label{conj:PQ}
For $D$ large enough, a primitive association scheme of $D$ classes
is $P$-polynomial if and only if it is $Q$-polynomial.
\end{conjecture}

We are not aware of any progress towards its proof.
The discovery of a feasible set of parameters of counter-examples
(see~\cite{MW})
casts some doubt on the conjecture,
and in the very least shows that this will likely be difficult to prove.
Moreover, the problem of classification of association schemes
which are both $P$- and $Q$-polynomial
(i.e., $Q$-polynomial distance-regular graphs) is still open.
We refer the reader to~\cite{DKT} for its current state.

Recall that, for a $P$-polynomial association scheme defined on a set $X$,
its intersection numbers $p_{ij}^k$
satisfy the {\em triangle inequality}: $p_{ij}^k=0$ if $|i-j|>k$ or $i+j<k$,
which naturally gives rise to a graph structure on $X$.
Perhaps, due to the lack of such an intuitive combinatorial characterization,
much less is known about $Q$-polynomial association schemes
when the $P$-polynomial property is absent
(which also indicates that there should be much more left to discover).
To date, only few examples of $Q$-polynomial schemes are known
which are neither $P$-polynomial nor duals of $P$-polynomial schemes~\cite{MT}
-- most of them are imprimitive and related to combinatorial designs.
The first infinite family of primitive $Q$-polynomial schemes that are not also $P$-polynomial 
was recently constructed in~\cite{PW}.
Due to Conjecture \ref{conj:PQ},
it seems that the most promising area for constructing new examples
of $Q$-polynomial association schemes which are not $P$-polynomial
includes those with few classes,
say, in the range $3\leq D\leq 6$.
The tables of feasible parameters
of primitive $3$-class $Q$-polynomial association schemes
and $4$- and $5$-class $Q$-bipartite association schemes
presented in Section~\ref{sec:tables}
may serve as a source for new constructions.

The parameters of $P$-polynomial association schemes are restricted by a number
of conditions implied by the triangle inequality.
On the other hand, the $Q$-polynomial property
allows us to consider {\em triple intersection numbers}
with respect to some triples of vertices,
which can be thought of as a generalization of intersection numbers
to triples of starting vertices instead of pairs.
This technique has been previously used by various researchers~%
\cite{CGS,CJ,GK,JKT,JV2012,JV2017,U,V},
mostly to prove nonexistence of some strongly regular
and distance-regular graphs with equality in the so-called Krein conditions,
in which case combining the restrictions implied by the triangle inequality
with triple intersection numbers seems the most fruitful.
Yet,
while calculating triple intersection numbers
when the $P$-polynomial property is absent is harder,
we managed to rule out a number of open cases from the tables.
This includes a putative $Q$-polynomial association scheme on $91$ vertices
whose existence has been open since 1999~\cite{vD}.

The paper is organized as follows.
In Section~\ref{sec:prelim},
we recall the basic theory of association schemes
and their triple intersection numbers.
In Section~\ref{sec:tables},
we comment on the tables of feasible parameters
of $Q$-polynomial association schemes and how they were generated.
In Section~\ref{sec:nonex},
we explain in details the analysis of triple intersection numbers
of $Q$-polynomial association schemes
and prove nonexistence for many open cases from the tables.
Finally, in Section~\ref{sec:quadruple},
we discuss the generalization of triple intersection numbers
to quadruples of vertices.

\section{Preliminaries}\label{sec:prelim}

In this section we prepare the notions needed in subsequent sections.

\subsection{Association schemes}\label{ssec:as}
Let $X$ be a finite set of vertices
and $\{R_0,R_1,\ldots,R_D\}$ be a set of non-empty subsets of $X\times X$.
Let $A_i$ denote the adjacency matrix of the (di-)graph $(X,R_i)$
($0 \le i \le D$).
The pair $(X,\{R_i\}_{i=0}^D)$ is called
a {\em (symmetric) association scheme} of $D$ classes
(or a {\em $D$-class scheme} for short)
if the following conditions hold:
\begin{enumerate}
\item $A_0 =I_{|X|}$, which is the identity matrix of size $|X|$,
\item $\sum_{i=0}^D A_i = J_{|X|}$,
which is the square all-one matrix of size $|X|$,
\item $A_i^\top=A_i$ ($1 \le i \le D$),
\item $A_iA_j=\sum_{k=0}^D p_{ij}^kA_k$,
where $p_{ij}^k$ are nonnegative integers  ($0 \le i,j \le D$).
\end{enumerate}
The nonnegative integers $p_{ij}^k$ are called {\em intersection numbers}:
for a pair of vertices $x, y\in X$ with $(x, y)\in R_k$
and integers $i$, $j$ ($0 \le i, j, k \le D$),
$p_{ij}^k$ equals the number of vertices $z \in X$ such that
$(x, z) \in R_i$, $(y, z) \in R_j$.

The vector space $\mathcal{A}$ over $\mathbb{R}$ spanned by the matrices
$A_i$ forms an algebra.
Since $\mathcal{A}$ is commutative and semisimple,
there exists a unique basis of $\mathcal{A}$ consisting of
primitive idempotents $E_0=\frac{1}{|X|}J_{|X|},E_1,\ldots,E_D$
(i.e., projectors onto the common eigenspaces of $A_0, \ldots, A_D$).
Since the algebra $\mathcal{A}$ is closed
under the entry-wise multiplication denoted by $\circ$,
we define the {\em Krein parameters}
$q_{ij}^k$ ($0 \le i,j,k \le D$) by
\begin{equation}\label{eqn:Kreinparameters}
  E_i\circ E_j=\frac{1}{|X|}\sum_{k=0}^D q_{ij}^kE_k.
\end{equation}

It is known that the Krein parameters are nonnegative real numbers
(see~\cite[Lemma 2.4]{D}).
Since both $\{A_0,A_1,\ldots,A_D\}$ and $\{E_0,E_1,\ldots,E_D\}$
form bases of $\mathcal{A}$,
there exists matrices $P=(P_{ij})_{i,j=0}^D$ and $Q=(Q_{ij})_{i,j=0}^D$ defined by
\begin{equation}\label{eqn:PQ}
  A_i=\sum_{j=0}^D P_{ji}E_j\text{~~and~~}E_i=\frac{1}{|X|}\sum_{j=0}^D Q_{ji}A_j.
\end{equation}
The matrices $P$ and $Q$ are called
the {\em first} and {\em second eigenmatrix} of $(X,\{R_i\}_{i=0}^D)$.

Let $n_i$, $0\leq i\leq D$, denote the {\em valency} of the graph $(X,R_i)$,
and $m_j$, $0\leq j\leq D$, denote the {\em multiplicity} of the eigenspace of $A_0,\ldots,A_D$
corresponding to $E_j$. Note that $n_i=p_{ii}^0$, while $m_j=q_{jj}^0$.

For an association scheme $(X,\{R_i\}_{i=0}^D)$,
an ordering of $A_1,\ldots,A_D$
such that for each $i$ ($0 \le i \le D$),
there exists a polynomial $v_i(x)$ of degree $i$
with $P_{ji}=v_i(P_{j1})$ ($0 \le j \le D$),
is called a {\em $P$-polynomial ordering} of relations.
An association scheme is said to be {\em $P$-polynomial}
if it admits a $P$-polynomial ordering of relations.
The notion of an association scheme
together with a $P$-polynomial ordering of relations
is equivalent to the notion of a {\em distance-regular graph}
-- such a graph has adjacency matrix $A_1$,
and $A_i$ ($0 \le i \le D$) is the adjacency matrix of its distance-$i$ graph
(i.e., $(x, y) \in R_i$ precisely when
$x$ and $y$ are at distance $i$ in the graph),
and the number of classes equals the diameter of the graph.
It is also known that an ordering of relations is $P$-polynomial
if and only if the matrix of intersection numbers $L_1$,
where $L_i:=(p_{ij}^k)_{k,j=0}^D$ ($0 \le i \le D$),
is a tridiagonal matrix with nonzero superdiagonal and subdiagonal~%
\cite[p.~189]{BI}
-- then $p_{ij}^k = 0$ holds whenever the triple $(i, j, k)$
does not satisfy the triangle inequality
(i.e., when $|i-j| < k$ or $i+j > k$).
For a $P$-polynomial ordering of relations of an association scheme,
set $a_i=p_{1,i}^i$, $b_i=p_{1,i+1}^i$, and $c_i=p_{1,i-1}^i$.
These intersection numbers are usually gathered
in the {\em intersection array}
$\{b_0, b_1, \dots, b_{D-1}; c_1, c_2, \dots, c_D\}$,
as the remaining intersection numbers can be computed from them
(in particular, $a_i = b_0 - b_i - c_i$ for all $i$, where $b_D = c_0 = 0$).
For an association scheme with a $P$-polynomial ordering of relations,
the ordering $E_1, \ldots, E_D$ is called
the {\em natural ordering} of eigenspaces
if $(P_{i1})_{i=0}^D$ is a decreasing sequence.

Dually, for an association scheme $(X,\{R_i\}_{i=0}^D)$,
an ordering of $E_1,\ldots,E_D$
such that for each $i$ ($0 \le i \le D$),
there exists a polynomial $v_i^*(x)$ of degree $i$
with $Q_{ji}=v_i^*(Q_{j1})$ ($0 \le j \le D$),
is called a {\em $Q$-polynomial ordering} of eigenspaces.
An association scheme is said to be {\em $Q$-polynomial}
if it admits a $Q$-polynomial ordering of eigenspaces.
Similarly as before,
it is known that an ordering of eigenspaces is $Q$-polynomial
if and only if the matrix of Krein parameters $L_1^*$,
where $L_i^*:=(q_{ij}^k)_{k,j=0}^D$ ($0 \le i \le D$),
is a tridiagonal matrix with nonzero superdiagonal and subdiagonal~%
\cite[p.~193]{BI}
-- then $q_{ij}^k = 0$ holds whenever the triple $(i, j, k)$
does not satisfy the triangle inequality.
For a $Q$-polynomial ordering of eigenspaces,
set $a_i^*=q_{1,i}^i$, $b_i^*=q_{1,i+1}^i$, and $c_i^*=q_{1,i-1}^i$.
Again, these Krein parameters are usually gathered in the {\em Krein array}
$\{b_0^*, b_1^*, \dots, b_{D-1}^*; c_1^*, c_2^*, \dots, c_D^*\}$
containing all the information needed to compute
the remaining Krein parameters
(in particular, we have $a_i^* = b_0^* - b_i^* - c_i^*$ for all $i$,
where $b_D^* = c_0^* = 0$).
For an association scheme with a $Q$-polynomial ordering of eigenspaces,
the ordering $A_1, \ldots, A_D$ is called
the {\em natural ordering} of relations
if $(Q_{i1})_{i=0}^D$ is a decreasing sequence.
Unlike for the $P$-polynomial association schemes,
there is no known general combinatorial characterization
of $Q$-polynomial association schemes.

An association scheme is called {\em primitive}
if all of $A_1, \ldots, A_D$ are adjacency matrices of connected graphs.
It is known that a distance-regular graph is imprimitive
precisely when it is a cycle of composite length,
an antipodal graph, or a bipartite graph
(possibly more than one of these),
see~\cite[Thm.~4.2.1]{BCN}.
The last two properties can be recognised from the intersection array
as $b_i = c_{D-i}$ ($0 \le i \le D$, $i \ne \lfloor D/2 \rfloor$)
and $a_i = 0$ ($0 \le i \le D$), respectively.
We may define dual properties for a $Q$-polynomial association scheme
-- we say that it is {\em $Q$-antipodal}
if $b_i^* = c_{D-i}^*$ ($0 \le i \le D$, $i \ne \lfloor D/2 \rfloor$),
and {\em $Q$-bipartite} if $a_i^* = 0$ ($0 \le i \le D$).
All imprimitive $Q$-polynomial association schemes
are schemes of cycles of composite length,
$Q$-antipodal or $Q$-bipartite (again, possibly more than one of these).
The original classification theorem by Suzuki~\cite{Suz}
allowed two more cases,
which have however been ruled out later~\cite{CS,TT}.
An association scheme that is both $P$- and $Q$-polynomial
is $Q$-antipodal if and only if it is bipartite,
and is $Q$-bipartite if and only if it is antipodal.

A {\em formal dual} of an association scheme
with first and second eigenmatrices $P$ and $Q$
is an association scheme such that,
for some orderings of its relations and eigenspaces,
its first and second eigenmatrices are $Q$ and $P$, respectively.
Note that this duality occurs on the level of parameters
-- an association scheme might have several formal duals, or none at all
(we can speak of duality when there exists
a regular abelian group of automorphisms, see~\cite[\S 2.10B]{BCN}).
An association scheme with $P = Q$
for some orderings of its relations and eigenspaces
is called {\em formally self-dual}.
For such orderings, $p^k_{ij} = q^k_{ij}$ ($0 \le i, j, k \le D$) holds
-- in particular, a formally self-dual association scheme
is $P$-polynomial if and only if it is $Q$-polynomial,
and then its intersection array matches its Krein array.

Any primitive association scheme with two classes is both $P$- and $Q$-polynomial
for either of the two orderings of relations and eigenspaces.
The graph with adjacency matrix $A_1$ of such a scheme
is said to be {\em strongly regular} (an {\it SRG} for short) with parameters $\nklm$,
where $n = |X|$ is the number of vertices,
$k = p^0_{11}$ is the valency of each vertex,
and each two distinct vertices
have precisely $\lambda = p^1_{11}$ common neighbours if they are adjacent,
and $\mu = p^2_{11}$ common neighbours if they are not adjacent.
In the sequel, we will identify $P$-polynomial association schemes
with their corresponding strongly regular or distance-regular graphs.

By a {\em parameter set} of an association scheme,
we mean the full set of $p_{ij}^k$, $q_{ij}^k$, $P_{ij}$ and $Q_{ij}$
described in this section, which are real numbers satisfying
the identities in~\cite[Lemma~2.2.1, Lemma~2.3.1]{BCN}.
We say that a parameter set for an association scheme is {\em feasible}
if it passes all known condition for the existence
of a corresponding association scheme.
For distance-regular graphs, there are many known feasibility conditions,
see~\cite{BCN,DKT,V}.
For $Q$-polynomial association schemes, much less is known
-- see Section~\ref{sec:tables} for the feasibility conditions we have used.

\subsection{Triple intersection numbers}\label{ssec:triple}
For a triple of vertices $x, y, z \in X$
and integers $i$, $j$, $k$ ($0 \le i, j, k \le D$)
we denote by $\sV{x & y & z}{i & j & k}$
(or simply $[i\ j\ k]$ when it is clear
which triple $(x, y, z)$ we have in mind)
the number of vertices $w \in X$ such that
$(x, w) \in R_i$, $(y, w) \in R_j$ and $(z, w) \in R_k$.
We call these numbers {\em triple intersection numbers}.

Unlike the intersection numbers,
the triple intersection numbers depend, in general,
on the particular choice of $(x, y, z)$.
Nevertheless, for a fixed triple $(x, y, z)$,
we may write down a system of $3D^2$ linear Diophantine equations
with $D^3$ triple intersection numbers as variables,
thus relating them to the intersection numbers, cf.~\cite{JV2012}:
{\small
\begin{equation}
\sum_{\ell=0}^D [\ell\ j\ k] = p^t_{jk}, \qquad
\sum_{\ell=0}^D [i\ \ell\ k] = p^s_{ik}, \qquad
\sum_{\ell=0}^D [i\ j\ \ell] = p^r_{ij},
\label{eqn:triple}
\end{equation}
}
where $(x, y) \in R_r$, $(x, z) \in R_s$, $(y, z) \in R_t$, and
\[
[0\ j\ k] = \delta_{jr} \delta_{ks}, \qquad
[i\ 0\ k] = \delta_{ir} \delta_{kt}, \qquad
[i\ j\ 0] = \delta_{is} \delta_{jt}.
\]
Moreover, the following theorem sometimes gives additional equations.

\begin{theorem}\label{thm:krein0}{\rm (\cite[Theorem~3]{CJ},
                                   cf.~\cite[Theorem~2.3.2]{BCN})}
Let $(X, \{R_i\}_{i=0}^D)$ be an association scheme of $D$ classes
with second eigenmatrix $Q$
and Krein parameters $q_{rs}^t$ $(0 \le r,s,t \le D)$.
Then,
\[
\pushQED{\bqed}
q_{rs}^t = 0 \quad \Longleftrightarrow \quad
\sum_{i,j,k=0}^D Q_{ir}Q_{js}Q_{kt}\sV{x & y & z}{i & j & k} = 0
\quad \mbox{for all\ } x, y, z \in X. \qedhere
\popQED
\]
\end{theorem}

Note that in a $Q$-polynomial association scheme,
many Krein parameters are zero,
and we can use Theorem~\ref{thm:krein0}
to obtain an equation for each of them.

\section{Tables of feasible parameters
for $Q$-polynomial association schemes}\label{sec:tables}

In this section we will describe the tables of feasible
parameter sets for primitive $3$-class $Q$-polynomial schemes
and $4$- and $5$-class $Q$-bipartite schemes.

These tables were all completed using the MAGMA programming
language (see~\cite{Magma}).
Any parameter set meeting the following conditions was included in the table:

\begin{enumerate}

\item  The parameters satisfy the $Q$-polynomial condition.

\item  All $p_{ij}^k$ are nonnegative integers, all valencies $p_{jj}^0$ are positive.

\item  For each $j > 0$ we have $n p_{jj}^0$ is even (the handshaking lemma applied to the graph $(X,R_j)$).

\item  For each $j,k > 0$ we have $p_{jj}^0 p_{jk}^j $ is even (the handshaking lemma applied
to the subconstituent $(\{y\in X\mid (x,y)\in R_j\},R_k)$, $x\in X$).

\item  For each $j > 0$ we have $n p_{jj}^0 (p_{jj}^0-1)$ is divisible by 6 (the number of triangles
in each graph $(X,R_j)$ is integral).

\item  All $q_{ij}^k$ are nonnegative
and for each $j$ the multiplicity $q_{jj}^0$
(i.e., the dimension of the $E_j$-eigenspace)
is a positive integer (see~\cite[Proposition~2.2.2]{BCN}).

\item  For all $i,j$ we have
$ \sum\limits_{ q_{ij}^k \neq 0 } m_k \leq m_i m_j$ if $i\ne j$
and $ \sum\limits_{ q_{ii}^k \neq 0 } m_k \leq \frac{m_i(m_i-1)}{2}$
(the absolute bound,
see~\cite[Theorem~2.3.3]{BCN} and the references therein).

\item  The splitting field is at most a degree 2 extension of the rationals
(see~\cite{MW2}).

\end{enumerate}

We note that there are many other conditions known for the special case of distance-regular graphs.
It was decided to apply these conditions after the construction of the table, and those not meeting
these extra conditions were labelled as nonexistent with a note as to the condition not met.
We leave as an open question whether if any of these conditions could be generalized to any cases
beyond distance-regular graphs; this (perhaps faint) hope is the main reason that they are included in the table.

We begin with the tables for $Q$-bipartite schemes, since this case is somewhat simpler than the primitive case.
Schemes which are $Q$-bipartite are formally dual to bipartite distance-regular graphs.
As a consequence,
the formal dual to~\cite[Theorem~4.2.2(i)]{BCN}
gives the Krein array for the quotient scheme of a $Q$-bipartite scheme
(see~\cite{MMW}).
Namely, if the scheme has Krein array
$\{ b_0^*, b_1^*,\dots, b_{D-1}^* ; c_1^*, \dots, c_D^* \}$ and $q_{11}^2 = \mu^*$,
then the Krein array of the quotient is:
\[
\left\{ \frac{b_0^*b_1^*}{\mu^*}, \frac{b_2^*b_3^*}{\mu^*}, \dots,
    \frac{b_{2m-2}^* b_{2m-1}^*}{\mu^*};
\frac{c_{1}^*c_{2}^*}{\mu^*}, \frac{c_{3}^*c_{4}^*}{\mu^*}, \dots,
\frac{c_{2m-1}^*c_{2m}^*}{\mu^*} \right\},
\]
where $m = \lfloor \frac{D}{2} \rfloor$.  When $D=4,5$ we obtain $m=2$, so the quotient structure
is a strongly regular graph.
A database of strongly regular graph parameters up to $5000$ vertices
can be generated very quickly.
From there,
we note that the quotient scheme has multiplicities $1, m_2, m_4$,
and that $m_1+m_3 = 1 + m_2 + m_4$ for $4$-class
and $m_1+m_3+m_5 = 1 + m_2 + m_4$ for $5$-class schemes.
Using the identities of~\cite[Lemma~2.3.1]{BCN},
it is easily seen the Krein arrays
for $Q$-bipartite $4$- and $5$-class association schemes are
\begin{small}
\begin{gather*}
\begin{multlined}
\left\{ m_1, m_1-1, \frac{m_1(m_2-m_1+1)}{m_2},
    \frac{m_1(m_3-m_2+m_1-1)}{m_3}; \right. \\
    \left. 1, \frac{m_1(m_1-1)}{m_2}, \frac{m_1(m_2-m_1+1)}{m_3}, m_1 \right\}
\end{multlined}
\intertext{and}
\begin{multlined}
\left\{ m_1, m_1\mm 1, \frac{m_1(m_2\mm m_1\pp 1)}{m_2},
    \frac{m_1(m_3\mm m_2\pp m_1\mm 1)}{m_3},
    \frac{m_1(m_4\mm m_3\pp m_2\mm m_1\pp 1)}{m_4};
    \right. \\
\left. 1, \frac{m_1(m_1\mm 1)}{m_2}, \frac{m_1(m_2\mm m_1\pp 1)}{m_3},
    \frac{m_1(m_3\mm m_2\pp m_1\mm 1)}{m_4},  m_1 \right\}
\end{multlined}
\end{gather*}
\end{small}

From this it is clear that the multiplicities determine all the parameters of the scheme.

In the 4-class case, the parameters are entirely determined by the quotient's multiplicities
(with a chosen $Q$-polynomial ordering) and $m_1$.  To search, we take a strongly regular graph
parameter set, choose one of two possible orderings for its multiplicities, calling its multiplicities
$m_0 = 1$, $m_2$, $m_4$.  From the absolute bound, we have $1+m_2 \leq \frac{m_1(m_1+1)}{2}$, and
from the positivity of $c_2^*$ we have $\frac{(m_2-m_1+1)m_1}{m_2} \geq 0$.  We then search over
all $\sqrt{2(1+m_2)} - \frac{1}{2} \leq m_1 \leq m_2$, checking the conditions above.
Given that we are iterating over SRG parameters together with two orderings and one integer,
this search is very fast.  The limitation of the table to $10000$ vertices is mainly readability
and practicality.  The third author has unpublished tables (without comments or details)
to $100000$ vertices, and could probably go much further without trouble.

We note that $Q$-bipartite schemes with 5-classes are very similar,
except we must iterate over both $m_1$ and $m_3$.  Again, this is
a very quick search, but the relative scarcity of 5-class parameter
sets makes listing up to $50000$ vertices, with annotation, manageable.
The table actually goes slightly higher, to $50520$ vertices,
because of the existence of an example on that number of vertices.

The trickiest search was the primitive 3-class $Q$-polynomial parameter sets.
In this case, there is no non-trivial quotient scheme to build on.

We use the following observation.

\begin{theorem}
A primitive $Q$-polynomial association scheme of $3$ classes
must have a matrix $L_i$ with $4$ distinct eigenvalues.
\end{theorem}
\begin{proof}
Assume not.  If a matrix $A_i$ has only two distinct eigenvalues, it is either complete,
contradicting the fact that it is a 3-class scheme, or a disjoint union of more than one
complete graph, contradicting the fact the scheme is primitive.
Therefore, the only case left to consider is when $A_1, A_2, A_3$ all have three distinct eigenvalues.
(We note in passing that this implies that the corresponding graphs are all strongly regular.
In this case, the scheme would be called ``amorphic'',
for more on amorphic schemes see~\cite{IIM}).

After reordering the eigenspaces and noting that the $P$-matrix is nonsingular, we are left with the following for $P$:
\[
\begin{pmatrix}
1 & n_1 & n_2 & n_3 \\
1 & a & c & e  \\
1 & a  & d & f \\
1 & b & c & f
\end{pmatrix},
\]
where $n = 1+n_1+n_2+n_3$, and $d = b+c-a$, $e = -1-a-c$.  Solving for $Q = nP^{-1}$ we obtain:
\[
\begin{pmatrix}
1 & \frac{n_3+(n-1)f}{a-b} & \frac{n_2+(n-1)c}{a-b} & \frac{n_1+(n-1)a}{a-b} \\
1 & \frac{n_3-f}{a-b} & \frac{n_2-c}{a-b} & \frac{-n+n_1-a}{a-b}  \\
1 & \frac{n_3-f}{a-b} & \frac{-n+n_2-c}{a-b} & \frac{n_1-a}{a-b} \\
1 & \frac{-n+n_3-f}{a-b} & \frac{n_2-c}{a-b} & \frac{n_1-a}{a-b}
\end{pmatrix}.\]

We arrive at our final contradiction: since each column contains a repeated entry,
the $Q$-matrix cannot be generated by one column via polynomials.
\end{proof}

We note that, in fact, all $Q$-polynomial $D$-class schemes must have a relation
with $D+1$ distinct eigenvalues.  However, the above theorem and its proof is sufficient
for our needs.

From this we conclude that each 3-class primitive $Q$-polynomial scheme has an adjacency matrix,
which we label $A_1$, which has four distinct eigenvalues.  Then the corresponding $4 \times 4$
intersection matrix $L_1$ has four distinct eigenvalues.
From this matrix, all of the other parameters may be determined.
In particular,  from~\cite[Proposition~2.2.2]{BCN},
the left-eigenvectors of $L_1$, normalized so their leftmost entry is $1$,
must be the rows of $P$.

The rest of the parameters can be derived from the equations:
\begin{eqnarray*}
  L_j &=& P^{-1} \diag( P_{0j}, P_{1j}, \dots, P_{Dj} ) \, P, \\
  L_j^* &=& Q^{-1} \diag( Q_{0j}, Q_{1j}, \dots, Q_{Dj} ) \, Q.
\end{eqnarray*}

However,
checking the $Q$-polynomial condition
is done before the computation of all parameters.
We use the following theorem, a proof of which can be found in~\cite{MW}.

\begin{theorem}  Let $L_i$ be an intersection matrix of a $D$-class association scheme,
where $L_i$ has exactly $D+1$ distinct eigenvalues.  Then the scheme is $Q$-polynomial
if and only if there is a Vandermonde matrix $U$ such that $U^{-1} L_i U=T$
where $T$ is upper triangular.
\end{theorem}

It is not hard to show that without loss of generality we can take $T_{01}$ to be 0, implying that the first column of $U$ is an eigenvector of $L_1$.  We only then need to iterate over the three (nontrivial) eigenvectors of $L_1$ to check this condition.
If the $Q$-polynomial condition is met, the rest of the parameters are computed and checked for the above conditions.

The schemes are then split into cases, depending on whether there is a strongly regular graph as a relation, and whether the splitting field is rational or not:
\begin{enumerate}

\item  Diameter $3$ distance-regular graphs (DRG for short).

\item  No diameter $3$ DRG, there is a strongly regular graph as a relation,
the splitting field is the rational field.

\item  No diameter $3$ DRG, there is a strongly regular graph as a relation,
the splitting field is a degree-$2$ extension of the rational field.

\item  No diameter $3$ DRG, there is no strongly regular graph as a relation,
the splitting field is the rational field.

\item  No diameter $3$ DRG, there is no strongly regular graph as a relation,
the splitting field is a degree-$2$ extension of the rational field.

\end{enumerate}

We note that we do not have any examples of primitive,
3-class $Q$-polynomial schemes with an irrational splitting field,
but there are open parameter sets of such.
It would be interesting to determine if these exist.
We also point out that all the feasible parameter sets known
to us have rational Krein parameters.

\medskip

{\bf Case 1.}
For DRG's, we iterated over the number of vertices,
intersection array and valencies.
The order was $n$, $b_0=n_1$, $b_1$, $n_2$
(noting $n_2$ is a divisor of $n_1 b_1$),
then $b_2$
(noting $b_2$ must be a multiple of $\frac{n_3}{\gcd(n_2,n_3)}$,
where $n_3 = n-n_1-n_2$),
from which the rest could be determined.

When there is no DRG, it is tempting to try to formally dualize the above process.  However, the Krein parameters of a scheme do not have to be integral, or even rational.  For this reason, it seemed more advantageous to iterate over parameters that needed to be integral, namely the parameters $p_{ij}^k$.  All arithmetic was done in MAGMA using the rational field, or a splitting field of a degree two irreducible polynomial over the rationals.  Floating point arithmetic was avoided to minimize numerical errors.

\medskip

For the rest of the cases, $L_1$ and the valencies were iterated over.  In particular, the parameters $a=p_{12}^1, b=p_{13}^1$ and $c = p_{13}^2$, together with $n, n_1, n_2$ determine the rest of $L_1$, noting that $a+b \leq n_1-1$ and $c \leq n_1 - \frac{n_1a}{n_2}$.  Any matrix without 4 distinct eigenvalues or with an irreducible cubic factor in its characteristic polynomial was discarded.

\medskip

{\bf Cases 2 and 3.}
In these cases, we iterate over strongly regular graphs first,
with parameters $\nklm$.
We choose $A_3$ to be the adjacency matrix
of the strongly regular graph relation,
and $L_1$, $L_2$ to be fissions of the complement.
Given this, the choice of $n_1$ will determine $n_2$.
The possibilities for $n_1$ can be narrowed
by observing that $p_{33}^1$ = $\mu$,
$n_3 = k$ and $ p_{33}^1 n_1 = p_{13}^3 n_3$,
implying that $n_1$ is divisible by $\frac{n_3}{ \gcd(n_3,\mu)}$.

Using similar identities,
we find $b$ is divisible by $\frac{n_3}{\gcd(n_1,n_3)}$,
$a$ is divisible by $\frac{n_2}{\gcd(n_1,n_2)}$,
and $c = \frac{n_1 (n_3-b-\mu) }{n_2}$.
After choosing these parameters all of $L_1$ follows.

\medskip

{\bf Cases 4 and 5.}
In these cases,
we know $L_1, L_2$ and $L_3$ all have $4$ distinct eigenvalues.
Therefore, we can assume $n_1$ is the smallest valency, and that $a \leq b$.
Using $a$ is divisible by $\frac{n_2}{\gcd(n_1,n_2)}$,
$b$ is divisible by $\frac{n_3}{\gcd(n_1,n_3)}$, and $n_2$ divides $an_1$,
we choose $n_1, a, n_2, b, c$, from which the rest is determined.
This is the slowest part of the search,
and the reason the primitive table goes to $2800$ vertices.

\medskip

We close with some comments on the irrational splitting field case.
The $2$-class primitive $Q$-polynomial case
is equivalent to (primitive) strongly regular graphs.
The only case where strongly regular graphs
have an irrational splitting field is the so-called ``half-case'',
when the graph has valency $\frac{n-1}{2}$.
Such graphs do exist,
for example the Paley graphs for non-square prime powers $q$
with $q$ congruent to $1$ modulo $4$.
We note that no primitive $Q$-polynomial schemes with more than $2$ classes
and a quadratic splitting field are known.
All feasible parameter sets we know of are $3$-class
and have a strongly regular graph relation (case 3).
The corresponding strongly regular graphs are also all unknown
(see~\cite{Bsrg}).
We have no feasible parameter set for case 5.
However, one case 5 parameter set satisfied all criteria
except the handshaking lemma.
It is listed below, though not included in the online table.
Given this, we expect feasible parameter sets for case 5 to exist,
but may be quite large.
\begin{small}
\[
P \ee \begin{pmatrix}
1 & 285 & 285 & 405\\
1 & \!19\pp 8\sqrt{19}\! & \!-38\pp 1\sqrt{19}\! & \!18\mm 9\sqrt{19}\!\\
1 & -3 & 5 & -3\\
1 & \!19\mm 8\sqrt{19}\! & \!-38\mm 1\sqrt{19}\! & \!18\pp 9\sqrt{19}\!\\
\end{pmatrix}\!, \
Q \ee \begin{pmatrix}
1 & 60 & 855 & 60\\
1 & \frac{76+32\sqrt{19}}{19} & -9 & \frac{76-32\sqrt{19}}{19}\\
1 & \!\frac{-152+4\sqrt{19}}{19}\! & 15 & \!\frac{-152-4\sqrt{19}}{19}\!\\
1 & \frac{8-4\sqrt{19}}{3} & \frac{-19}{3} & \frac{8+4\sqrt{19}}{3}\\
\end{pmatrix}\!,
\]
\[
L_1 \ee \begin{pmatrix}
0 & 285 & 0 & 0\\
1 & 116 & 60 & 108\\
0 & 60 & 90 & 135\\
0 & 76 & 95 & 114\\
\end{pmatrix}, \quad
L_2 \ee \begin{pmatrix}
0 & 0 & 285 & 0\\
0 & 60 & 90 & 135\\
1 & 90 & 59 & 135\\
0 & 95 & 95 & 95\\
\end{pmatrix}, \quad
L_3 \ee \begin{pmatrix}
0 & 0 & 0 & 405\\
0 & 108 & 135 & 162\\
0 & 135 & 135 & 135\\
1 & 114 & 95 & 195\\
\end{pmatrix},
\]
\begin{align*}
  L_1^* &= \begin{pmatrix}
0 & 60 & 0 & 0\\
1 & \frac{400+32\sqrt{19}}{61} & \frac{3199-32\sqrt{19}}{61} & 0\\
0 & \frac{12796-128\sqrt{19}}{3477} & \frac{181184+128\sqrt{19}}{3477} & \frac{80}{19}\\
0 & 0 & 60 & 0\\
\end{pmatrix},\\
L_2^* &= \begin{pmatrix}
0 & 0 & 855 & 0\\
0 & \frac{3199-32\sqrt{19}}{61} & \frac{45296+32\sqrt{19}}{61} & 60\\
1 & \frac{181184+128\sqrt{19}}{3477} & \frac{137210}{183} & \frac{181184-128\sqrt{19}}{3477}\\
0 & 60 & \frac{45296-32\sqrt{19}}{61} & \frac{3199+32\sqrt{19}}{61}\\
\end{pmatrix},\\
L_3^* &= \begin{pmatrix}
0 & 0 & 0 & 60\\
0 & 0 & 60 & 0\\
0 & \frac{80}{19} & \frac{181184-128\sqrt{19}}{3477} & \frac{12796+128\sqrt{19}}{3477}\\
1 & 0 & \frac{3199+32\sqrt{19}}{61} & \frac{400-32\sqrt{19}}{61}\\
\end{pmatrix}.
\end{align*}
\end{small}

While feasible parameters may exist,
the complete lack of examples elicits the following question:
\begin{question}
Do all $3$-class primitive $Q$-polynomial schemes have a rational splitting field?
\end{question}

This is a special case
of the so-called ``Sensible Caveman'' conjecture of William J.~Martin:
\begin{conjecture}
For $Q$-polynomial schemes of $3$ or more classes,
if the scheme is primitive then its splitting field is rational.
\end{conjecture}

\section{Nonexistence results}\label{sec:nonex}

We derived our nonexistence results
by analyzing triple intersection numbers
of $Q$-poly\-no\-mi\-al association schemes.
For some choice of relations $R_r, R_s, R_t$,
the system of Diophantine equations
derived from \eqref{eqn:triple} and Theorem~\ref{thm:krein0}
may have multiple nonnegative solutions,
each giving the possible values of the triple intersection numbers
with respect to a triple $(x, y, z)$
with $(x, y) \in R_r$, $(x, z) \in R_s$ and $(y, z) \in R_t$.
However, in certain cases, there might be no nonnegative solutions
-- in this case,
we may conclude that an association scheme with the given parameters
does not exist.

Even when there are solutions for all choices of $R_r, R_s, R_t$
such that $p^t_{rs} \ne 0$,
sometimes nonexistence can be derived by other means.
We may, for example, employ double counting.

\begin{proposition}\label{prop:dblcnt}
Let $x$ and $y$ be vertices of an association scheme with $(x, y) \in R_r$.
Suppose that $\alpha_1, \alpha_2, \dots, \alpha_m$ are distinct integers
such that there are precisely $\kappa_\ell$ vertices $z$
with $(x, z) \in R_s$, $(y, z) \in R_t$
and $\sV{x & y & z}{i & j & k} = \alpha_\ell$
$(1 \le \ell \le m$, $\sum_{\ell=1}^m \kappa_\ell = p^r_{st})$,
and $\beta_1, \beta_2, \dots, \beta_n$ are distinct integers
such that there are precisely $\lambda_\ell$ vertices $w$
with $(w, x) \in R_i$, $(w, y) \in R_j$
and $\sV{w & x & y}{k & s & t} = \beta_\ell$
$(1 \le \ell \le n$, $\sum_{\ell=1}^n \lambda_\ell = p^r_{ij})$.
Then,
$$
\sum_{\ell=1}^m \kappa_\ell \alpha_\ell =
\sum_{\ell=1}^n \lambda_\ell \beta_\ell .
$$
\end{proposition}

\begin{proof}
Count the number of pairs $(w, z)$
with $(x, z) \in R_s$, $(y, z) \in R_t$, $(w, x) \in R_i$, $(w, y) \in R_j$
and $(w, z) \in R_k$.
\end{proof}

We consider the special case of Proposition~\ref{prop:dblcnt}
when a triple intersection number is zero for all triples of vertices
in some given relations.

\begin{corollary}\label{cor:dblcnt}
Suppose that for all vertices $x, y, z$ of an association scheme
with $(x, y) \in R_r$, $(x, z) \in R_s$, $(y, z) \in R_t$,
$\sV{x & y & z}{i & j & k} = 0$ holds.
Then, $\sV{w & x & y}{k & s & t} = 0$ holds
for all vertices $w, x, y$
with $(w, x) \in R_i$, $(w, y) \in R_j$ and $(x, y) \in R_r$.
\end{corollary}

\begin{proof}
Apply Proposition~\ref{prop:dblcnt} to all $(x, y) \in R_r$,
with $m \le 1$ and $\alpha_1 = 0$.
Since $\beta_\ell$ and $\lambda_\ell$ ($1 \le \ell \le n$)
must be nonnegative,
it follows that $n \le 1$ and $\beta_1 = 0$.
\end{proof}

\subsection{Computer search}\label{ssec:comp}

The {\tt sage-drg} package~\cite{Vdrg, V} by the second author
for the SageMath computer algebra system~\cite{Sage}
has been used to perform computations
of triple intersection numbers of $Q$-polynomial association schemes
with Krein arrays that were marked as open
in the tables of feasible parameter sets by the third author~\cite{WQpoly},
see Section~\ref{sec:tables}.
The package was originally developed for the purposes of feasibility checking
for intersection arrays of distance-regular graphs
and included a routine to find general solutions to the system of equations
for computing triple intersection numbers.

For the purposes of the current research,
the package has been extended to support
parameters of general association schemes, in particular,
given as Krein arrays of $Q$-poly\-no\-mi\-al association schemes.
Additionally, the package now supports generating integral solutions
for systems of equations with constraints on the solutions
(e.g., nonnegativity of triple intersection numbers)
-- these can also be added on-the-fly.
The routine uses SageMath's mixed integer linear programming facilities,
which support multiple solvers.
We have used SageMath's default GLPK solver~\cite{GLPK}
and the CBC solver~\cite{CBC} in our computations
-- however, other solvers can also be used if they are available.

We have thus been able to implement an algorithm
which tries to narrow down the possible solutions of the systems of equations
for determining triple intersection numbers of an association scheme
such that they satisfy Corollary~\ref{cor:dblcnt},
and conclude inequality if any of the systems of equations
has no such feasible solutions.

\begin{enumerate}
\item
For each triple of relations $(R_r, R_s, R_t)$ such that $p^t_{rs} > 0$,
initialize an empty set of solutions,
obtain a general (i.e., parametric) solution to the system of equations
derived from \eqref{eqn:triple} and Theorem~\ref{thm:krein0},
and initialize a generator of solutions
with the constraint that the intersection numbers be integral and nonnegative.
All generators $(r, s, t)$ are initially marked as \pty{active},
and all triple intersection numbers $(r, s, t; i, j, k)$
(representing $\sV{x & y & z}{i & j & k}$
with $(x, y) \in R_r$, $(x, z) \in R_s$ and $(y, z) \in R_t$)
are initially marked as \pty{unknown}.

\item \label{item:next}
For each \pty{active} generator, generate one solution
and add it to the corresponding set of solutions.
If a generator does not return a new solution
(i.e., it has exhausted all of them),
then mark it as \pty{inactive}.

\item
For each \pty{inactive} generator,
verify that the corresponding set of solutions is non\-empty
-- otherwise, terminate and conclude nonexistence.

\item Initialize an empty set $Z$.

\item For each \pty{unknown} triple intersection number $(r, s, t; i, j, k)$,
mark it as \pty{nonzero} if a solution has been found
in which its value is not zero.
If such a solution has not been found yet,
make a copy of the generator $(r, s, t)$
with the constraint that $(r, s, t; i, j, k)$ be nonzero,
and generate one solution.
If such a solution exists,
add it to the set of solutions
and mark $(r, s, t; i, j, k)$ as \pty{nonzero},
otherwise mark $(r, s, t; i, j, k)$ as \pty{zero} and add it to $Z$.

\item If $Z$ is empty, terminate without concluding nonexistence.

\item For each triple intersection number $(r, s, t; i, j, k) \in Z$
and for each \pty{nonzero} $(a, b, c;$ $d, e, f) \in
\{(r, i, j; s, t, k), (s, i, k; r, t, j), (t, j, k; r, s, i)\}$,
remove all solutions from the corresponding set
in which the value of the latter is nonzero,
mark $(a, b, c; d, e, f)$ as \pty{zero},
mark all \pty{nonzero} $(a, b, c; \ell, m, n)$
with $(\ell, m, n) \ne (d, e, f)$ as \pty{unknown},
and add a constraint that $(a, b, c;$ $d, e, f)$ be zero
to the generator $(a, b, c)$ if it is \pty{active}.

\item Go to \ref{item:next}.
\end{enumerate}

Note that generators and triple intersection numbers
are considered equivalent under permutation of vertices,
i.e., under actions $(r, s, t) \mapsto (r, s, t)^\pi$
and $(r, s, t; i, j, k) \mapsto ((r, s, t)^\pi; (i, j, k)^{\pi^{(1\ 3)}})$
for $\pi \in S_3$.

The above algorithm is available as the \pythoninline{check_quadruples} method
of {\tt sage-drg}'s \pythoninline{ASParameters} class.
We ran it for all open cases in the tables from Section~\ref{sec:tables},
and obtained $29$ nonexistence results for primitive $3$-class schemes,
$92$ nonexistence results for $Q$-bipartite $4$-class schemes,
and $11$ nonexistence results for $Q$-bipartite $5$-class schemes.
The results are summarized in the following theorem.

\afterpage{
\begin{table}[t]
\begin{centering}
\begin{footnotesize}
\begin{tabular}{ccccc}
Label & Krein array & DRG & Nonexistence & Family \\
\hline
\willif{91,12} &
$\{12, {338 \over 35}, {39 \over 25}; 1, {312 \over 175}, {39 \over 5}\}$
&& $(3, 3, 3)$ & \\
\willif{225,24} & $\{24, 20, {36 \over 11}; 1, {30 \over 11}, 24\}$
&& $(3, 1, 1; 3, 3, 1)$ & \\
\willif{324,17} & $\{17, 16, 10; 1, 2, 8\}$
&& $(1, 1, 2)$ & \eqref{eqn:fam3} \\
\willif{324,19} & $\{19, {128 \over 9}, 10; 1, {16 \over 9}, 10\}$
&& $(1, 1, 3)$ & \\
\willif{441,20} & $\{20, {378 \over 25}, 12; 1, {42 \over 25}, 9\}$
&& $(1, 1, 3)$ & \\
\willif{540,33} & $\{33, 20, {63 \over 5}; 1, {12 \over 5}, 15\}$
&& $(1, 1, 3)$ & \\
\willif{540,35} &
$\{35, {243 \over 10}, {27 \over 2}; 1, {27 \over 10}, {45 \over 2}\}$
&& $(1, 1, 3)$ & \\
\willif{576,23} & $\{23, {432 \over 25}, 15; 1, {48 \over 25}, 9\}$
&& $(1, 1, 3)$ & \\
\willif{729,26} & $\{26, {486 \over 25}, 18; 1, {54 \over 25}, 9\}$
&& $(1, 1, 3)$ & \\
\willif{1000,37} & $\{37, 24, 14; 1, 2, 12\}$ && $(1, 1, 3)$ & \\
\willif{1015,28} &
$\{28, {2523 \over 130}, {4263 \over 338};
1, {1218 \over 845}, {203 \over 26}\}$
&& $(1, 1, 3)$ & \\
\willif{1080,83} & $\{83, 54, 21; 1, 6, 63\}$ & FSD & $(1, 1, 2)$ & \\
\willif{1134,49} &
$\{49, 48, {644 \over 75}; 1, {196 \over 75}, 42\}$ && $(1, 1, 1)$ & \\
\willif{1189,40} &
$\{40, {5043 \over 203}, {123 \over 7}; 1, {615 \over 406}, {164 \over 7}\}$
&& $(1, 1, 2)$ & \\
\willif{1470,104} & $\{104, 70, 25; 1, 7, 80\}$
& FSD & $(1, 1, 2; 3, 2, 3)$ & \\
\willif{1548,35} &
$\{35, {2187 \over 86}, {45 \over 2}; 1, {135 \over 86}, {27 \over 2}\}$
&& $(1, 1, 3)$ & \\
\willif[a]{1680,69} & $\{69, 42, 7; 1, 2, 63\}$ && $(1, 1, 2)$ & \\
\willif{1702,45} &
$\{45, {4761 \over 148}, {115 \over 4}; 1, {345 \over 148}, {69 \over 4}\}$
&& $(1, 1, 2)$ & \\
\willif{1944,29} & $\{29, 22, 25; 1, 2, 5\}$ && $(1, 1, 2)$ & \\
\willif{2016,195} & $\{195, 160, 28; 1, 20, 168\}$ & FSD & $(1, 2, 2)$ & \\
\willif{2106,65} & $\{65, 64, {676 \over 25}; 1, {104 \over 25}, 26\}$
& \{125, 108, 24; 1, 9, 75\} & $(1, 1, 1)$ & \\
\willif{2185,114} &
$\{114, {4761 \over 65}, {58121 \over 1521};
1, {11799 \over 1690}, {6118 \over 117}\}$
&& $(1, 1, 3)$ & \\
\willif{2197,36} &
$\{36, {45 \over 2}, {45 \over 2}; 1, {3 \over 2}, {15 \over 2}\}$
&& $(1, 1, 3)$ & \\
\willif{2197,126} & $\{126, 90, 10; 1, 6, 105\}$
& FSD (0231) & $(2, 2, 3)$ & \\
\willif{2304,47} & $\{47, {135 \over 4}, 33; 1, {9 \over 4}, 15\}$
&& $(1, 1, 3)$ & \\
\willif{2376,95} & $\{95, 63, 12; 1, 3, 84\}$ && $(1, 1, 3)$ & \\
\willif{2401,48} & $\{48, 30, 29; 1, {3 \over 2}, 20\}$ && $(1, 1, 2)$ & \\
\willif[a]{2500,49} & $\{49, 48, 26; 1, 2, 24\}$
&& $(1, 1, 2)$ & \eqref{eqn:fam3} \\
\willif{2640,203} & $\{203, 160, 34; 1, 16, 170\}$ & FSD & $(1, 2, 2)$ & \\
\end{tabular}
\caption{
Nonexistence results for feasible Krein arrays
of primitive $3$-class $Q$-polynomial association schemes
on up to $2800$ vertices.
For $P$-polynomial parameters
(for the natural ordering of relations, unless otherwise indicated),
the DRG column indicates whether the parameters are formally self-dual (FSD),
or the intersection array is given.
The Nonexistence column gives either the triple of relation indices
for which there is no solution for triple intersection numbers,
or the $6$-tuple of relation indices $(r, s, t; i, j, k)$
for which Corollary~\ref{cor:dblcnt} is not satisfied.
The Family column specifies the infinite family
from Subsection~\ref{ssec:families}
that the parameter set is part of.
}
\label{tab:d3prim}
\end{footnotesize}
\end{centering}
\end{table}
\clearpage
}

\afterpage{\aem 
\begin{center}
\begin{footnotesize}
\begin{tabular}[p]{cccc}
Label & Krein array & Nonexistence & Family \\
\hline
\willif{200,12} &
$\{12, 11, {256 \over 25}, {36 \over 11};
1, {44 \over 25}, {96 \over 11}, 12\}$
& $(1, 1, 2)$ & \\
\willif{462,21} &
$\{21, 20, {196 \over 11}, {49 \over 5};
1, {35 \over 11}, {56 \over 5}, 21\}$
& $(1, 1, 2)$ & \\
\willif{486,45} & $\{45, 44, 36, 5; 1, 9, 40, 45\}$
& $(1, 1, 2)$ & \\
\willif{578,17} & $\{17, 16, {136 \over 9}, 9; 1, {17 \over 9}, 8, 17\}$
& $(1, 1, 2)$ & \eqref{eqn:fam4} \\
\willif{686,28} & $\{28, 27, 25, 8; 1, 3, 20, 28\}$
& $(1, 2, 2)$ & \\
\willif{702,36} &
$\{36, 35, {405 \over 13}, {72 \over 7};
1, {63 \over 13}, {180 \over 7}, 36\}$
& $(1, 2, 2)$ & \\
\willif{722,19} & $\{19, 18, {152 \over 9}, 11; 1, {19 \over 9}, 8, 19\}$
& $(1, 1, 2)$ & \eqref{eqn:fam4} \\
\willif{882,21} & $\{21, 20, {56 \over 3}, 13; 1, {7 \over 3}, 8, 21\}$
& $(1, 1, 2)$ & \eqref{eqn:fam4} \\
\willif{990,66} &
$\{66, 65, {847 \over 15}, {88 \over 13};
1, {143 \over 15}, {770 \over 13}, 66\}$
& $(1, 2, 2)$ & \\
\willif{1014,78} & $\{78, 77, 65, 8; 1, 13, 70, 78\}$
& $(1, 2, 2)$ & \\
\willif{1058,23} & $\{23, 22, {184 \over 9}, 15; 1, {23 \over 9}, 8, 23\}$
& $(1, 1, 2)$ & \eqref{eqn:fam4} \\
\willif{1250,25} & $\{25, 24, {200 \over 9}, 17; 1, {25 \over 9}, 8, 25\}$
& $(1, 1, 2)$ & \eqref{eqn:fam4} \\
\willif{1458,27} & $\{27, 26, 24, 19; 1, 3, 8, 27\}$
& $(1, 1, 2)$ & \eqref{eqn:fam4} \\
\willif{1458,36} & $\{36, 35, 33, 16; 1, 3, 20, 36\}$
& $(1, 2, 2)$ & \\
\willif{1482,38} &
$\{38, 37, {12635 \over 351}, {76 \over 37};
1, {703 \over 351}, {1330 \over 37}, 38\}$
& $(1, 2, 2)$ & \\
\willif{1674,45} &
$\{45, 44, {1296 \over 31}, {135 \over 11};
1, {99 \over 31}, {360 \over 11}, 45\}$
& $(1, 1, 2)$ & \\
\willif{1682,29} & $\{29, 28, {232 \over 9}, 21; 1, {29 \over 9}, 8, 29\}$
& $(1, 1, 2)$ & \eqref{eqn:fam4} \\
\willif{1694,55} & $\{55, 54, {352 \over 7}, 15; 1, {33 \over 7}, 40, 55\}$
& $(1, 1, 2)$ & \\
\willif{1862,21} &
$\{21, 20, {364 \over 19}, {81 \over 5};
1, {35 \over 19}, {24 \over 5}, 21\}$
& $(1, 1, 2)$ & \\
\willif{2058,49} &
$\{49, 48, {686 \over 15}, {77 \over 5};
1, {49 \over 15}, {168 \over 5}, 49\}$
& $(1, 1, 2)$ & \\
\willif{2060,50} &
$\{50, 49, {4800 \over 103}, {110 \over 7};
1, {350 \over 103}, {240 \over 7}, 50\}$
& $(1, 1, 2)$ & \\
\willif{2394,27} &
$\{27, 26, {3240 \over 133}, {279 \over 13};
1, {351 \over 133}, {72 \over 13}, 27\}$
& $(1, 1, 2)$ & \\
\willif{2466,36} &
$\{36, 35, {4617 \over 137}, {144 \over 7};
1, {315 \over 137}, {108 \over 7}, 36\}$
& $(1, 2, 2)$ & \\
\willif{2550,85} &
$\{85, 84, {1156 \over 15}, {187 \over 7};
1, {119 \over 15}, {408 \over 7}, 85\}$
& $(1, 1, 2)$ & \\
\willif{2662,121} &
$\{121, 120, {5324 \over 49}, {77 \over 5};
1, {605 \over 49}, {528 \over 5}, 121\}$
& $(1, 1, 2)$ & \\
\willif{2706,66} &
$\{66, 65, {2541 \over 41}, {44 \over 3};
1, {165 \over 41}, {154 \over 3}, 66\}$
& $(1, 2, 2)$ & \\
\willif{2730,78} &
$\{78, 77, {507 \over 7}, {52 \over 3};
1, {39 \over 7}, {182 \over 3}, 78\}$
& $(1, 2, 2)$ & \\
\willif{2750,25} &
$\{25, 24, {250 \over 11}, {185 \over 9};
1, {25 \over 11}, {40 \over 9}, 25\}$
& $(1, 1, 2)$ & \\
\willif{2862,53} &
$\{53, 52, {11236 \over 225}, {265 \over 13};
1, {689 \over 225}, {424 \over 13}, 53\}$
& $(1, 1, 2)$ & \\
\willif{2890,153} & $\{153, 152, 136, 9; 1, 17, 144, 153\}$
& $(1, 1, 2)$ & \\
\willif{2926,171} &
$\{171, 170, {11552 \over 77}, {171 \over 17};
1, {1615 \over 77}, {2736 \over 17}, 171\}$
& $(1, 1, 2)$ & \\
\willif{2970,54} & $\{54, 53, {567 \over 11}, 12; 1, {27 \over 11}, 42, 54\}$
& $(1, 2, 2)$ & \\
\willif{3042,65} & $\{65, 64, {182 \over 3}, 25; 1, {13 \over 3}, 40, 65\}$
& $(1, 1, 2)$ & \\
\willif{3074,106} &
$\{106, 105, {2809 \over 29}, {212 \over 9};
1, {265 \over 29}, {742 \over 9}, 106\}$
& $(1, 2, 2)$ & \\
\willif{3174,184} & $\{184, 183, 161, 16; 1, 23, 168, 184\}$
& $(1, 2, 2)$ & \\
\willif{3250,50} &
$\{50, 49, {625 \over 13}, {100 \over 9};
1, {25 \over 13}, {350 \over 9}, 50\}$
& $(1, 2, 2)$ & \\
\willif{3402,126} &
$\{126, 125, {343 \over 3}, 28; 1, {35 \over 3}, 98, 126\}$
& $(1, 2, 2)$ & \\
\willif{3498,77} &
$\{77, 76, {3872 \over 53}, {231 \over 19};
1, {209 \over 53}, {1232 \over 19}, 77\}$
& $(1, 1, 2)$ & \\
\willif{3610,133} &
$\{133, 132, {608 \over 5}, 21; 1, {57 \over 5}, 112, 133\}$
& $(1, 1, 2)$ & \\
\willif{3726,36} & $\{36, 35, {783 \over 23}, 24; 1, {45 \over 23}, 12, 36\}$
& $(1, 2, 2)$ & \\
\willif{4070,55} &
$\{55, 54, {1936 \over 37}, {77 \over 3};
1, {99 \over 37}, {88 \over 3}, 55\}$
& $(1, 1, 2)$ & \\
\willif{4250,119} &
$\{119, 118, {13872 \over 125}, {1309 \over 59};
1, {1003 \over 125}, {5712 \over 59}, 119\}$
& $(1, 1, 2)$ & \\
\willif{4370,190} &
$\{190, 189, {3971 \over 23}, {76 \over 7};
1, {399 \over 23}, {1254 \over 7}, 190\}$
& $(1, 2, 2)$ & \\
\willif{4410,210} & $\{210, 209, 189, 12; 1, 21, 198, 210\}$
& $(1, 2, 2)$ & \\
\willif{4464,24} &
$\{24, 23, {2048 \over 93}, {488 \over 23};
1, {184 \over 93}, {64 \over 23}, 24\}$
& $(1, 1, 2)$ & \\
\willif{4526,73} &
$\{73, 72, {10658 \over 155}, {511 \over 15};
1, {657 \over 155}, {584 \over 15}, 73\}$
& $(1, 1, 2)$ & \\
\willif{4558,86} &
$\{86, 85, {12943 \over 159}, {1376 \over 51};
1, {731 \over 159}, {3010 \over 51}, 86\}$
& $(1, 2, 2)$ & \\
\willif{4590,75} &
$\{75, 74, {1200 \over 17}, 35; 1, {75 \over 17}, 40, 75\}$
& $(1, 1, 2)$ & \\
\multicolumn{4}{r}{{\footnotesize\em (Continued on next page.)}} \\
\end{tabular}
\end{footnotesize}
\end{center}

\clearpage
\begin{table}[t]
\begin{centering}
\begin{footnotesize}
\begin{tabular}{cccc}
\multicolumn{4}{l}{{\footnotesize\em (Continued.)}} \\
Label & Krein array & Nonexistence & Family \\
\hline
\willif{4758,117} &
$\{117, 116, {6760 \over 61}, {273 \over 29};
1, {377 \over 61}, {3120 \over 29}, 117\}$
& $(1, 1, 2)$ & \\
\willif{4802,49} &
$\{49, 48, {1176 \over 25}, 25; 1, {49 \over 25}, 24, 49\}$
& $(1, 1, 2)$ & \eqref{eqn:fam4} \\
\willif{5046,261} & $\{261, 260, 232, 21; 1, 29, 240, 261\}$
& $(1, 1, 2)$ & \\
\willif{5202,51} &
$\{51, 50, {1224 \over 25}, 27; 1, {51 \over 25}, 24, 51\}$
& $(1, 1, 2)$ & \eqref{eqn:fam4} \\
\willif{5480,100} &
$\{100, 99, {12800 \over 137}, {140 \over 3};
1, {900 \over 137}, {160 \over 3}, 100\}$
& $(1, 1, 2)$ & \\
\willif{5566,66} &
$\{66, 65, {1463 \over 23}, 24; 1, {55 \over 23}, 42, 66\}$
& $(1, 2, 2)$ & \\
\willif{5590,78} &
$\{78, 77, {3211 \over 43}, {312 \over 11};
1, {143 \over 43}, {546 \over 11}, 78\}$
& $(1, 2, 2)$ & \\
\willif{5618,53} &
$\{53, 52, {1272 \over 25}, 29; 1, {53 \over 25}, 24, 53\}$
& $(1, 1, 2)$ & \eqref{eqn:fam4} \\
\willif{5618,106} &
$\{106, 105, {901 \over 9}, 36; 1, {53 \over 9}, 70, 106\}$
& $(1, 2, 2)$ & \\
\willif{5642,91} &
$\{91, 90, {2704 \over 31}, {65 \over 3};
1, {117 \over 31}, {208 \over 3}, 91\}$
& $(1, 1, 2)$ & \\
\willif{5670,105} & $\{105, 104, 98, 49; 1, 7, 56, 105\}$
& $(1, 1, 2)$ & \\
\willif[a]{5670,105} & $\{105, 104, 100, 25; 1, 5, 80, 105\}$
& $(1, 1, 2)$ & \\
\willif{6050,55} & $\{55, 54, {264 \over 5}, 31; 1, {11 \over 5}, 24, 55\}$
& $(1, 1, 2)$ & \eqref{eqn:fam4} \\
\willif{6278,73} &
$\{73, 72, {21316 \over 301}, {365 \over 21};
1, {657 \over 301}, {1168 \over 21}, 73\}$
& $(1, 1, 2)$ & \\
\willif{6358,85} & $\{85, 84, {884 \over 11}, 45; 1, {51 \over 11}, 40, 85\}$
& $(1, 1, 2)$ & \\
\willif{6422,91} &
$\{91, 90, {1664 \over 19}, {119 \over 5};
1, {65 \over 19}, {336 \over 5}, 91\}$
& $(1, 1, 2)$ & \\
\willif{6426,147} &
$\{147, 146, {2352 \over 17}, 35; 1, {147 \over 17}, 112, 147\}$
& $(1, 1, 2)$ & \\
\willif{6450,105} &
$\{105, 104, {4320 \over 43}, {357 \over 13};
1, {195 \over 43}, {1008 \over 13}, 105\}$
& $(1, 1, 2)$ & \\
\willif{6498,57} &
$\{57, 56, {1368 \over 25}, 33; 1, {57 \over 25}, 24, 57\}$
& $(1, 1, 2)$ & \eqref{eqn:fam4} \\
\willif{6962,59} &
$\{59, 58, {1416 \over 25}, 35; 1, {59 \over 25}, 24, 59\}$
& $(1, 1, 2)$ & \eqref{eqn:fam4} \\
\willif{7210,103} &
$\{103, 102, {84872 \over 875}, {927 \over 17};
1, {5253 \over 875}, {824 \over 17}, 103\}$
& $(1, 1, 2)$ & \\
\willif{7442,61} &
$\{61, 60, {1464 \over 25}, 37; 1, {61 \over 25}, 24, 61\}$
& $(1, 1, 2)$ & \eqref{eqn:fam4} \\
\willif{7854,66} &
$\{66, 65, {1089 \over 17}, {88 \over 3};
1, {33 \over 17}, {110 \over 3}, 66\}$
& $(1, 2, 2)$ & \\
\willif{7878,78} &
$\{78, 77, {7605 \over 101}, {104 \over 3};
1, {273 \over 101}, {130 \over 3}, 78\}$
& $(1, 2, 2)$ & \\
\willif{7906,134} &
$\{134, 133, {22445 \over 177}, {2948 \over 57};
1, {1273 \over 177}, {4690 \over 57}, 134\}$
& $(1, 2, 2)$ & \\
\willif{7938,63} &
$\{63, 62, {1512 \over 25}, 39; 1, {63 \over 25}, 24, 63\}$
& $(1, 1, 2)$ & \eqref{eqn:fam4} \\
\willif{8120,100} &
$\{100, 99, {19200 \over 203}, {620 \over 11};
1, {1100 \over 203}, {480 \over 11}, 100\}$
& $(1, 1, 2)$ & \\
\willif{8190,90} &
$\{90, 89, {1125 \over 13}, 40; 1, {45 \over 13}, 50, 90\}$
& $(1, 2, 2)$ & \\
\willif{8246,217} &
$\{217, 216, {3844 \over 19}, {155 \over 3};
1, {279 \over 19}, {496 \over 3}, 217\}$
& $(1, 1, 2)$ & \\
\willif{8450,65} & $\{65, 64, {312 \over 5}, 41; 1, {13 \over 5}, 24, 65\}$
& $(1, 1, 2)$ & \eqref{eqn:fam4} \\
\willif{8450,78} & $\{78, 77, {377 \over 5}, 36; 1, {13 \over 5}, 42, 78\}$
& $(1, 2, 2)$ & \\
\willif{8470,88} & $\{88, 87, {429 \over 5}, 16; 1, {11 \over 5}, 72, 88\}$
& $(1, 2, 2)$ & \\
\willif{8478,27} &
$\{27, 26, {3888 \over 157}, {327 \over 13};
1, {351 \over 157}, {24 \over 13}, 27\}$
& $(1, 1, 2)$ & \\
\willif{8750,325} & $\{325, 324, 300, 13; 1, 25, 312, 325\}$
& $(1, 1, 2)$ & \\
\willif{8758,232} &
$\{232, 231, {32799 \over 151}, {464 \over 11};
1, {2233 \over 151}, {2088 \over 11}, 232\}$
& $(1, 2, 2)$ & \\
\willif{8798,106} &
$\{106, 105, {8427 \over 83}, {424 \over 9};
1, {371 \over 83}, {530 \over 9}, 106\}$
& $(1, 2, 2)$ & \\
\willif{8802,351} &
$\{351, 350, {52488 \over 163}, {351 \over 25};
1, {4725 \over 163}, {8424 \over 25}, 351\}$
& $(1, 1, 2)$ & \\
\willif{8978,67} &
$\{67, 66, {1608 \over 25}, 43; 1, {67 \over 25}, 24, 67\}$
& $(1, 1, 2)$ & \eqref{eqn:fam4} \\
\willif{9310,105} &
$\{105, 104, {17500 \over 171}, {165 \over 13};
1, {455 \over 171}, {1200 \over 13}, 105\}$
& $(1, 1, 2)$ & \\
\willif{9350,153} &
$\{153, 152, {8092 \over 55}, {459 \over 19};
1, {323 \over 55}, {2448 \over 19}, 153\}$
& $(1, 1, 2)$ & \\
\willif{9386,171} &
$\{171, 170, {2128 \over 13}, 27; 1, {95 \over 13}, 144, 171\}$
& $(1, 1, 2)$ & \\
\willif{9522,69} &
$\{69, 68, {1656 \over 25}, 45; 1, {69 \over 25}, 24, 69\}$
& $(1, 1, 2)$ & \eqref{eqn:fam4} \\
\willif{9522,161} &
$\{161, 160, {460 \over 3}, 49; 1, {23 \over 3}, 112, 161\}$
& $(1, 1, 2)$ & \\
\willif{9702,126} &
$\{126, 125, {1323 \over 11}, 56; 1, {63 \over 11}, 70, 126\}$
& $(1, 2, 2)$ & \\
\end{tabular}
\caption{
Nonexistence results for feasible Krein arrays
of $Q$-bipartite (but not $Q$-antipodal)
$4$-class $Q$-polynomial association schemes
on up to $10000$ vertices.
The Nonexistence column gives either the triple of relation indices
for which there is no solution for triple intersection numbers.
The Family column specifies the infinite family
from Subsection~\ref{ssec:families}
that the parameter set is part of.
}
\label{tab:d4qbip}
\end{footnotesize}
\end{centering}
\end{table}
\clearpage
}

\afterpage{
\begin{table}[t]
\begin{centering}
\begin{footnotesize}
\begin{tabular}{cccc}
Label & Krein array & Family \\
\hline
\willif{576,21} &
$\{21, 20, 18, {21 \over 2}, {27 \over 7};
1, 3, {21 \over 2}, {120 \over 7}, 21\}$ & \\
\willif{800,25} &
$\{25, 24, {625 \over 28}, {75 \over 7}, {25 \over 7};
1, {75 \over 28}, {100 \over 7}, {150 \over 7}, 25\}$
& \eqref{eqn:fam5} \\
\willif{2000,25} &
$\{25, 24, {625 \over 27}, {50 \over 3}, {25 \over 9};
1, {50 \over 27}, {25 \over 3}, {200 \over 9}, 25\}$ & \\
\willif{2400,22} &
$\{22, 21, 20, {88 \over 5}, {32 \over 11};
1, 2, {22 \over 5}, {210 \over 11}, 22\}$ & \\
\willif{2928,61} &
$\{61, 60, {3721 \over 66}, {305 \over 11}, {61 \over 11};
1, {305 \over 66}, {366 \over 11}, {610 \over 11}, 61\}$
& \eqref{eqn:fam5} \\
\willif{7232,113} &
$\{113, 112, {12769 \over 120}, {791 \over 15}, {113 \over 15};
1, {791 \over 120}, {904 \over 15}, {1582 \over 15}, 113\}$
& \eqref{eqn:fam5} \\
\willif{14480,181} &
$\{181, 180, {32761 \over 190}, {1629 \over 19}, {181 \over 19};
1, {1629 \over 190}, {1810 \over 19}, {3258 \over 19}, 181\}$
& \eqref{eqn:fam5} \\
\willif{25440,265} &
$\{265, 264, {70225 \over 276}, {2915 \over 23}, {265 \over 23};
1, {2915 \over 276}, {3180 \over 23}, {5830 \over 23}, 265\}$
& \eqref{eqn:fam5} \\
\willif{37752,121} &
$\{121, 120, {14641 \over 125}, {484 \over 5}, {121 \over 25};
1, {484 \over 125}, {121 \over 5}, {2904 \over 25}, 121\}$ & \\
\willif{40880,365} &
$\{365, 364, {133225 \over 378}, {4745 \over 27}, {365 \over 27};
1, {4745 \over 378}, {5110 \over 27}, {9490 \over 27}, 365\}$
& \eqref{eqn:fam5} \\
\willif{47040,116} &
$\{116, 115, 112, {696 \over 7}, {144 \over 29};
1, 4, {116 \over 7}, {3220 \over 29}, 116\}$ & \\
\end{tabular}
\caption{
Nonexistence results for feasible Krein arrays
of $Q$-bipartite (but not $Q$-antipodal)
$5$-class $Q$-polynomial association schemes
on up to $50000$ vertices.
In all cases,
there is no solution for triple intersection numbers
for a triple of vertices mutually in relation $R_1$.
The Family column specifies the infinite family
from Subsection~\ref{ssec:families}
that the parameter set is part of.
}
\label{tab:d5qbip}
\end{footnotesize}
\end{centering}
\end{table}
}

\begin{theorem}\label{thm:nonex}
A $Q$-polynomial association scheme with Krein array listed
in one of Tables~\ref{tab:d3prim}, \ref{tab:d4qbip} and~\ref{tab:d5qbip}
does not exist.
\end{theorem}

\begin{proof}
In all but two cases,
it suffices to observe that for some triple of relations $R_r, R_s, R_t$,
the system of equations
derived from \eqref{eqn:triple} and Theorem~\ref{thm:krein0}
has no integral nonnegative solutions
-- Tables~\ref{tab:d3prim} and~\ref{tab:d4qbip} list the triple $(r, s, t)$,
while for all examples in Table~\ref{tab:d5qbip},
this is true for $(r, s, t) = (1, 1, 1)$.
Note that the natural ordering of the relations is used.

Let us now consider the cases
\willif{225,24} and \willif{1470,104} from Table~\ref{tab:d3prim}.
In the first case, the Krein array is $\{24, 20, 36/11; 1, 30/11, 24\}$.
Such an association scheme has two $Q$-polynomial orderings,
so we can augment the system of equations \eqref{eqn:triple}
with six equations derived from Theorem~\ref{thm:krein0}.
Let $w, x, y, z$ be vertices such that
$(x, z), (y, z) \in R_1$ and $(w, x), (w, y), (x, y) \in R_3$.
Since $p^3_{11} = 22$ and $p^3_{33} = 3$, such vertices must exist.
We first compute the triple intersection numbers with respect to $x, y, z$.
There are two integral nonnegative solutions,
both having $[3\ 3\ 1] = 0$.
On the other hand, there is a single solution
for the triple intersection numbers with respect to $w, x, y$,
giving $[1\ 1\ 1] = 3$.
However, this contradicts Corollary~\ref{cor:dblcnt},
so such an association scheme does not exist.

In the second case, the Krein array is $\{104, 70, 25; 1, 7, 80\}$.
Let $w, x, y, z$ be vertices such that $(x, y), (x, z) \in R_1$,
$(w, y), (y, z) \in R_2$ and $(w, x) \in R_3$.
Since $p^1_{12} = 70$ and $p^1_{32} = 250$, such vertices must exist.
There is a single solution
for the triple intersection numbers with respect to $x, y, z$,
giving $[3\ 2\ 3] = 0$.
On the other hand, there are four solutions
for the triple intersection numbers with respect to $w, x, y$,
from which we obtain $[3\ 1\ 2] \in \{15, 16, 17, 18\}$.
Again, this contradicts Corollary~\ref{cor:dblcnt},
so such an association scheme does not exist.
This completes the proof.
\end{proof}

\begin{remark}
The {\tt sage-drg} package repository provides two Jupyter notebooks
containing the computation details in the proofs of nonexistence
of two cases from Table~\ref{tab:d3prim}:
\begin{itemize}
\item
\href{https://nbviewer.jupyter.org/github/jaanos/sage-drg/blob/master/jupyter/QPoly-24-20-36_11-1-30_11-24.ipynb}{\tt QPoly-24-20-36\_11-1-30\_11-24.ipynb}
for the case \willif{225,24}, and
\item
\href{https://nbviewer.jupyter.org/github/jaanos/sage-drg/blob/master/jupyter/DRG-104-70-25-1-7-80.ipynb}{\tt DRG-104-70-25-1-7-80.ipynb}
for the case \willif{1470,104}.
\end{itemize}
\end{remark}

\begin{remark}
The parameter set \willif{91,12} from Table~\ref{tab:d3prim}
was listed by Van Dam~\cite{vD}
as the smallest feasible $Q$-polynomial parameter set
for which no scheme is known.
The next such open case is now
the Krein array $\{14, 108/11, 15/4; 1, 24/11, 45/4\}$
for a primitive $3$-class $Q$-polynomial association scheme
with $99$ vertices,
which was also listed by Van Dam.
\end{remark}

Since some of the parameters from Table~\ref{tab:d3prim}
also admit a $P$-polynomial ordering,
we can derive nonexistence of distance-regular graphs
with certain intersection arrays.
We have also found an intersection array
for a primitive $Q$-polynomial distance-regular graph of diameter $4$,
which is listed in~\cite{BCN} and~\cite{Bdrg},
and for which, to the best of our knowledge,
nonexistence has not been previously known.

\begin{theorem}\label{thm:nonex-drg}
There is no distance-regular graph with intersection array
\begin{align*}
\{83, 54, 21&; 1, 6, 63\}, \\
\{104, 70, 25&; 1, 7, 80\}, \\
\{195, 160, 28&; 1, 20, 168\}, \\
\{125, 108, 24&; 1, 9, 75\}, \\
\{126, 90, 10&; 1, 6, 105\}, \\
\{203, 160, 34&; 1, 16, 170\}, \text{or} \\
\{53, 40, 28, 16&; 1, 4, 10, 28\}.
\end{align*}
\end{theorem}

\begin{proof}
The cases \willif{1080,83}, \willif{1470,104}, \willif{2016,195}
and \willif{2640,203} from Table~\ref{tab:d3prim}
are formally self-dual for the natural ordering of relations,
while \willif{2197,126} is formally self-dual with ordering of relations
$A_2, A_3, A_1$ relative to the natural ordering.
In each case, the corresponding association scheme is $P$-polynomial
with intersection array equal to the Krein array.
The case \willif{2106,65} is not formally self-dual,
yet the natural ordering of relations is $P$-polynomial
with intersection array $\{125, 108, 24; 1, 9, 75\}$.
In all of the above cases,
Theorem~\ref{thm:nonex} implies nonexistence
of the corresponding association scheme,
so a distance-regular graph with such an intersection array does not exist.

Consider now a distance-regular graph with intersection array
$\{53, 40, 28, 16; 1, 4, 10,$ $28\}$.
Such a graph is formally self-dual for the natural ordering of eigenspaces
and therefore also $Q$-polynomial.
Augmenting the system of equations \eqref{eqn:triple}
with twelve equations derived from Theorem~\ref{thm:krein0}
gives a two parameter solution for triple intersection numbers
with respect to three vertices mutually at distances $1, 3, 3$.
However, it turns out that there is no integral solution,
leading to nonexistence of the graph.
\end{proof}

\begin{remark}
The non-existence of a distance-regular graph
with intersection array $\{53, 40,$ $28, 16;$ $1, 4, 10, 28\}$
also follows by applying the Terwilliger polynomial~\cite{GK}.
Recall that this polynomial, say $T_{\Gamma}(x)$, which depends only on the intersection numbers 
of a $Q$-polynomial distance-regular graph $\Gamma$ and its $Q$-polynomial ordering, satisfies:
\begin{equation}\label{eq:Tpoly}
T_{\Gamma}(\eta)\geq 0,
\end{equation}
where $\eta$ is any non-principal eigenvalue of the local graph of an arbitrary vertex $x$ of $\Gamma$.
Furthermore, by~\cite[Theorem~4.4.3(i)]{BCN}, $\eta$ satisfies
\begin{equation}\label{eq:443}
-1-\frac{b_1}{\theta_1+1}\leq \eta\leq -1-\frac{b_1}{\theta_D+1}, 
\end{equation}
where $b_0=\theta_0>\theta_1>\ldots>\theta_D$ are the $D+1$ distinct eigenvalues of $\Gamma$.

For the above-mentioned intersection array, $T_{\Gamma}(x)$ is a polynomial of degree $4$ with 
a negative leading term and the following roots: 
$-\frac{7}{3}$ $(=-1-\frac{b_1}{\theta_1+1})$,
$\frac{9-\sqrt{249}}{4}\approx -1.695$,
$\frac{17}{3}$ $(=-1-\frac{b_1}{\theta_D+1})$,
$\frac{9+\sqrt{249}}{9}\approx 6.195$.
 
Thus, combining \eqref{eq:Tpoly} and \eqref{eq:443}, we obtain
\[
-\frac{7}{3}\leq \eta\leq \frac{9-\sqrt{249}}{4}\text{~or~} \eta=\frac{17}{3},
\]
and one can finally obtain a contradiction as in~\cite[Claim~4.3]{GKbf}.
\end{remark}

\subsection{Infinite families}\label{ssec:families}

The data from Tables~\ref{tab:d3prim}, \ref{tab:d4qbip} and~\ref{tab:d5qbip}
allows us to look for infinite families of Krein arrays
for which we can show nonexistence
of corresponding $Q$-polynomial association schemes.
We find three families, one for each number of classes.

The first family of Krein arrays is given by
\begin{equation}\label{eqn:fam3}
\{2r^2-1, 2r^2-2, r^2+1; 1, 2, r^2-1\}.
\end{equation}
This Krein array is feasible for all integers $r \ge 2$.
A $Q$-polynomial association scheme with Krein array \eqref{eqn:fam3}
has $3$ classes and $4r^4$ vertices.
Examples exist when $r$ is a power of $2$
-- they are realized by duals of Kasami codes with minimum distance $5$,
see~\cite[\S 11.2]{BCN}.

\begin{theorem}\label{thm:fam3}
A $Q$-polynomial association scheme
with Krein array \eqref{eqn:fam3} and $r$ odd does not exist.
\end{theorem}

\begin{proof}
Consider a $Q$-polynomial association scheme
with Krein array \eqref{eqn:fam3}.
Besides the Krein parameters failing the triangle inequality,
$q^1_{11}$ is also zero.
Therefore, in order to compute triple intersection numbers,
the system of equations \eqref{eqn:triple}
can be augmented with four equations derived from Theorem~\ref{thm:krein0}.
We compute triple intersection numbers with respect to vertices $x, y, z$
such that $(x, y), (x, z) \in R_1$ and $(y, z) \in R_2$.
Since $p^2_{11} = r(r+2)(r^2-1)/4 > 0$, such vertices must exist.
We obtain a four parameter solution
(see the notebook
\href{https://nbviewer.jupyter.org/github/jaanos/sage-drg/blob/master/jupyter/QPoly-d3-1param-odd.ipynb}{\tt QPoly-d3-1param-odd.ipynb}
on the {\tt sage-drg} package repository for computation details).
Then we may express
$$
[1\ 2\ 3] = -{r^4 \over 2} + 2r^2 + [1\ 3\ 1] + 3 \cdot [2\ 3\ 3]
            - [3\ 1\ 1] + 4 \cdot [3\ 3\ 3] .
$$
Clearly, the above triple intersection number can only be integral
when $r$ is even.
Therefore, we conclude that a $Q$-polynomial association scheme
with Krein array \eqref{eqn:fam3} and $r$ odd does not exist.
\end{proof}

The next family is a two parameter family of Krein arrays
\begin{equation}\label{eqn:fam4}
\{m, m-1, m(r^2-1)/r^2, m-r^2+1; 1, m/r^2, r^2-1, m\}
\end{equation}
This Krein array is feasible for all integers $m$ and $r$
such that $0 < 2(r^2-1) \le m \le r(r-1)(r+2)$ and $m(r+1)$ is even.
A $Q$-polynomial association scheme with Krein array \eqref{eqn:fam4}
is $Q$-bipartite with $4$ classes and $2m^2$ vertices.
One may take the $Q$-bipartite quotient of such a scheme
(i.e., identify vertices in relation $R_4$)
to obtain a strongly regular graph with parameters
$\nklm = (m^2, (m-1)r^2, m + r^2(r^2-3), r^2(r^2-1))$,
i.e., a pseudo-Latin square graph.
Therefore, we say that a scheme with Krein array \eqref{eqn:fam4}
is of {\em Latin square type}.

There are several examples of $Q$-polynomial association schemes
with Krein array \eqref{eqn:fam4} for some $r$ and $m$.
For $(r, m) = (2, 6)$ and $(r, m) = (3, 16)$,
this Krein array is realized by the schemes of shortest vectors
of the $E_6$ lattice
and an overlattice of the Barnes-Wall lattice in $\R^{16}$~\cite{MT},
respectively.
For $(r, m) = (2^{ij}, 2^{i(2j+1)})$,
there are examples arising from duals of extended Kasami codes~%
\cite[\S 11.2]{BCN}
for each choice of positive integers $i$ and $j$.
In particular, the Krein array obtained by setting $i = j = 1$
uniquely determines the halved $8$-cube.

In the case when $r$ is a prime power and $m = r^3$,
the formal dual of this parameter set
(i.e., a distance-regular graph with the corresponding intersection array)
is realized by a Pasechnik graph~\cite{BP}.

\begin{theorem}\label{thm:fam4}
A $Q$-polynomial association scheme
with Krein array \eqref{eqn:fam4} and $m$ odd does not exist.
\end{theorem}

\begin{proof}
Consider a $Q$-polynomial association scheme
with Krein array \eqref{eqn:fam4}.
Since the scheme is $Q$-bipartite,
we have $q^k_{ij} = 0$ whenever $i+j+k$ is odd
or the triple $(i, j, k)$ does not satisfy the triangle inequality.
This allows us to augment the system of equations \eqref{eqn:triple}
with many equations derived from Theorem~\ref{thm:krein0}.
We compute triple intersection numbers with respect to vertices $x, y, z$
such that $(x, y), (x, z) \in R_1$ and $(y, z) \in R_2$.
Since $p^2_{11} = r^2(r^2-1)/2 > 0$, such vertices must exist.
We obtain a one parameter solution
(see the notebook
\href{https://nbviewer.jupyter.org/github/jaanos/sage-drg/blob/master/jupyter/QPoly-d4-LS-odd.ipynb}{\tt QPoly-d4-LS-odd.ipynb}
on the {\tt sage-drg} package repository for computation details)
which allows us to express
$$
[1\ 1\ 3] = r + {r^2(1-r) \over 2} - {m \over 2} + [1\ 1\ 1] .
$$
Clearly, the above triple intersection number can only be integral
when $m$ is even.
Therefore, we conclude that a $Q$-polynomial association scheme
with Krein array \eqref{eqn:fam4} and $m$ odd does not exist.
\end{proof}

The last family is given by the Krein array
\begin{equation}\label{eqn:fam5}
\begin{multlined}
\left\{{r^2+1 \over 2}, {r^2-1 \over 2}, {(r^2+1)^2 \over 2r(r+1)},
{(r-1)(r^2+1) \over 4r}, {r^2+1 \over 2r}; \right. \\
\left. 1, {(r-1)(r^2+1) \over 2r(r+1)}, {(r+1)(r^2 + 1) \over 4r},
{(r-1)(r^2+1) \over 2r}, {r^2+1 \over 2}\right\}
\end{multlined}
\end{equation}
This Krein array is feasible for all odd $r \ge 5$.
A $Q$-polynomial association scheme with Krein array \eqref{eqn:fam5}
is $Q$-bipartite with $5$ classes and $2(r+1)(r^2+1)$ vertices.
One may take the $Q$-bipartite quotient of such a scheme
to obtain a strongly regular graph with parameters
$\nklm = ((r+1)(r^2+1), r(r+1), r-1, r+1)$
-- these are precisely the parameters of collinearity graphs
of generalized quadrangles $\GQ(r, r)$.
The scheme also has a second $Q$-polynomial ordering of eigenspaces,
namely the ordering $E_5, E_2, E_3, E_4, E_1$
relative to the ordering implied by the Krein array.
For $r \equiv 1 \pmod{4}$ a prime power,
the Krein array \eqref{eqn:fam5} is realized by a scheme
derived by Moorhouse and Williford~\cite{MW}
from a double cover of the $C_2(r)$ dual polar graph.

\begin{theorem}\label{thm:fam3}
A $Q$-polynomial association scheme
with Krein array \eqref{eqn:fam5} and $r \equiv 3 \pmod{4}$
does not exist.
\end{theorem}

\begin{proof}
Consider a $Q$-polynomial association scheme
with Krein array \eqref{eqn:fam5}.
Since the scheme is $Q$-bipartite,
we have $q^k_{ij} = 0$ whenever $i+j+k$ is odd
or the triple $(i, j, k)$ does not satisfy the triangle inequality.
This allows us to augment the system of equations \eqref{eqn:triple}
with many equations derived from Theorem~\ref{thm:krein0}.
We compute triple intersection numbers with respect to vertices $x, y, z$
that are mutually in relation $R_1$.
Since $p^1_{11} = (r-1)/2 > 0$, such vertices must exist.
We obtain a single solution
(see the notebook
\href{https://nbviewer.jupyter.org/github/jaanos/sage-drg/blob/master/jupyter/QPoly-d5-1param-3mod4.ipynb}{\tt QPoly-d5-1param-3mod4.ipynb}
on the {\tt sage-drg} package repository for computation details)
with
$$
[1\ 1\ 1] = {r-5 \over 4} .
$$
Clearly, the above triple intersection number can only be integral
when $r \equiv 1 \pmod{4}$.
Therefore, we conclude that a $Q$-polynomial association scheme
with Krein array \eqref{eqn:fam5} and $r \equiv 3 \pmod{4}$ does not exist.
\end{proof}

\section{Quadruple intersection numbers}\label{sec:quadruple}

The argument of the proof of Theorem~\ref{thm:krein0}
(\cite[Theorem~2.3.2]{BCN})
can be further extended to $s$-tuples of vertices
(see Remark (iii) in~\cite[\S2.3]{BCN}).
In particular, we may consider {\em quadruple intersection numbers} with respect
to a quadruple of vertices $w, x, y, z\in X$.
For integers $h, i, j, k$ ($0 \le h, i, j, k \le D$),
denote by $\sW{w & x & y & z}{h & i & j & k}$
(or simply $[h\ i\ j\ k]$ when it is clear
which quadruple $(w, x, y, z)$ we have in mind)
the number of vertices $u \in X$ such that
$(u, w) \in R_h$, $(u, x) \in R_i$, $(u, y) \in R_j$, and $(u, z) \in R_k$.

For a fixed quadruple $(w, x, y, z)$,
one can obtain a system of linear Diophantine equations
with quadruple intersection numbers as variables
which relates them to the intersection numbers
(or to the triple intersection numbers).

The following analogue of Theorem~\ref{thm:krein0} allows us
to obtain some additional equations.

\begin{theorem}\label{thm:doublekrein0}
Let $(X, \{R_i\}_{i=0}^D)$ be an association scheme of $D$ classes
with second eigenmatrix $Q$
and Krein parameters $q_{ij}^k$ $(0 \le i,j,k \le D)$.
Then, for fixed $p, r, s, t$ $(0 \le p, r, s, t \le D)$,
\[
\sum_{\ell=0}^Dq^\ell_{pr}q^\ell_{st}=0 \quad \Longleftrightarrow \
\sum_{h,i,j,k=0}^D Q_{hp} Q_{ir} Q_{js} Q_{kt}
\sW{w & x & y & z}{h & i & j & k} = 0
\quad \text{for all $w, x, y, z \in X$.}
\]
\end{theorem}
\begin{proof}
Since $E_i$ is a symmetric idempotent matrix, one has
\begin{equation}\label{eqn:idemp}
\sum_{w \in X} E_i(u, w) E_i(v, w) = E_i(u, v).
\end{equation}

Let $\Sigma(M)$ denote the sum of all entries of a matrix $M$.
Then, by \eqref{eqn:idemp},
\begin{alignat}{2}
\nonumber
\Sigma(E_p \circ E_r \circ E_s \circ E_t) &=&
\sum_{u, v \in X} \ \ &{} E_p(u, v) E_r(u, v) E_s(u, v) E_t(u, v) \\
\nonumber &=& \sum_{w, x, y, z \in X}
  &{} \left(\sum_{u\in X} E_p(u, w) E_r(u, x) E_s(u, y) E_t(u, z)\right)
  \cdot \\
\nonumber
&&&{} \left(\sum_{v\in X} E_p(v, w) E_r(v, x) E_s(v, y) E_t(v, z)\right) \\
\label{eqn:quad1}
&=&{} \sum_{w, x, y, z \in X} &{} \sigma(w, x, y, z)^2\geq 0,
\end{alignat}
where
$\sigma(w, x, y, z) = \sum_{u\in X} E_p(u, w) E_r(u, x) E_s(u, y) E_t(u, z)$.

On the other hand, by \eqref{eqn:Kreinparameters},
\begin{align}
\nonumber |X|^2 \, \Sigma(E_p \circ E_r \circ E_s \circ E_t)
&= |X|^2 \, \Tr( (E_p \circ E_r) \cdot (E_s \circ E_t) ) \\
\nonumber
&= \Tr\left(\left(\sum_{\ell=0}^D q^\ell_{pr} E_\ell \right)\cdot
            \left(\sum_{\ell=0}^D q^\ell_{st} E_\ell\right)\right) \\
\label{eqn:quad2}
&= \sum_{\ell=0}^D m_\ell q^\ell_{pr} q^\ell_{st},
\end{align}
where $m_\ell$ is the rank of $E_\ell$
(i.e., the multiplicity of the corresponding eigenspace),
and by \eqref{eqn:PQ},
\begin{align}
\nonumber |X|^3 \, \Sigma(E_p \circ E_r \circ E_s \circ E_t)
&= \frac{1}{|X|} \sum_{\ell=0}^D
    Q_{\ell p} Q_{\ell r} Q_{\ell s} Q_{\ell t} \Sigma(A_\ell) \\
\label{eqn:quad3}
&= \sum_{\ell=0}^D n_\ell Q_{\ell p} Q_{\ell r} Q_{\ell s} Q_{\ell t},
\end{align}
where $n_\ell$ is the valency of $(X, R_\ell)$.

Since the multiplicities $m_\ell$ are positive numbers
and the Krein parameters are non-negative numbers,
by \eqref{eqn:quad1}, \eqref{eqn:quad2}, \eqref{eqn:quad3},
we have $\Sigma(E_p \circ E_r \circ E_s \circ E_t) = 0$
if and only if $q^\ell_{pr} q^\ell_{st} = 0$
(with fixed $p, r, s, t$) for all $\ell = 0, \ldots, D$.
In this case, we have $\sigma(w, x, y, z)=0$
for all quadruples $(w, x, y, z)$,
which implies
\begin{align*}
0 = |X|^4 \, \sigma(w, x, y, z)
&= |X|^4 \, \sum_{u \in X} E_p(u, w) E_r(u, x) E_s(u, y) E_t(u, z) \\
&= \sum_{h,i,j,k=0}^D Q_{hp} Q_{ir} Q_{js} Q_{kt}
    \sW{w & x & y & z}{h & i & j & k} ,
\end{align*}
which completes the proof.
\end{proof}

The condition of Theorem \ref{thm:doublekrein0} is satisfied when,
for example, an association scheme is $Q$-bipartite,
i.e., $q_{ij}^k=0$ whenever $i+j+k$ is odd
(take $p+r$ and $s+t$ of different parity).

Suda~\cite{S} lists several families of association schemes
which are known to be {\em triply regular},
i.e., their triple intersection numbers $\sV{x & y & z}{i & j & k}$
only depend on $i, j, k$ and the mutual distances between $x, y, z$,
and not on the choices of the vertices themselves:
\begin{itemize}
  \item strongly regular graphs with $q_{11}^1=0$ (cf.~\cite{CGS}),
  \item Taylor graphs (antipodal $Q$-bipartite schemes of $3$ classes),
  \item linked systems of symmetric designs
    (certain $Q$-antipodal schemes of $3$ classes) with $a_1^*=0$,
  \item tight spherical $7$-designs
    (certain $Q$-bipartite schemes of $4$ classes), and
  \item collections of real mutually unbiased bases
    ($Q$-antipodal $Q$-bipartite schemes of $4$ classes).
\end{itemize}
Schemes belonging to the above families seem natural candidates
for the computations of their quadruple intersection numbers.
However, the condition of Theorem~\ref{thm:doublekrein0}
is never satisfied for primitive strongly regular graphs,
while for Taylor graphs the obtained equations do not give any information
that could not be obtained
through relating the quadruple intersection numbers
to the triple intersection numbers.
This was also the case for the examples
of triply regular linked systems of symmetric designs that we have checked.
However,
in the cases of tight spherical $7$-designs and mutually unbiased bases,
we do get new restrictions on quadruple intersection numbers.
So far,
we have not succeeded in using this new information
for either new constructions or proofs of nonexistence.

\Acknowledgements

\bibliographystyle{abbrv}
\bibliography{references}

\end{document}